\documentclass[12pt,a4paper]{article}
\usepackage{amsmath,amsthm,amsfonts,amssymb,bbm}
\usepackage{graphicx,psfrag,subfigure,color}
\usepackage{cite}
\usepackage{hyperref}

\numberwithin{equation}{section}

\newcommand{\eps}{\varepsilon}
\newcommand{\Or}{\mathcal{O}}

\newcommand{\Pb}{\mathbb{P}}
\newcommand{\E}{\mathbbm{E}}
\newcommand{\Id}{\mathbbm{1}}
\newcommand{\e}{\varepsilon}
\newcommand{\I}{{\rm i}}

\newcommand{\R}{\mathbb{R}}

\newcommand{\Z}{\mathbb{Z}}
\renewcommand{\Re}{\mathrm{Re}}

\newtheorem{prop}{Proposition}[section]
\newtheorem{thm}[prop]{Theorem}
\newtheorem{lem}[prop]{Lemma}

\newtheorem{cor}[prop]{Corollary}

\newtheorem{cla}[prop]{Claim}
\newtheorem{assumpt}[prop]{Assumption}

\newtheorem{rem}[prop]{Remark}
\newenvironment{remark}{\begin{rem}\normalfont}{\end{rem}}

\title{Limit distributions for KPZ growth models with spatially homogeneous random initial conditions}
\author{S. Chhita\thanks{Department of Mathematical Sciences, Durham University, Stockton Road, Durham, DH1 3LE, UK. E-mail: {\tt sunil.chhita@durham.ac.uk}} \and
P.L. Ferrari\thanks{Institute for Applied Mathematics, Bonn University, Endenicher Allee 60, 53115 Bonn, Germany. E-mail: {\tt ferrari@uni-bonn.de}} \and H. Spohn\thanks{Zentrum Mathematik, TU M\"unchen, Boltzmannstrasse 3, D-85747 Garching, Germany. E-mail: {\tt spohn@ma.tum.de}}}
\date{April 7, 2017}

\begin{document}
\sloppy
\maketitle

\begin{abstract}
For stationary KPZ growth in $1+1$ dimensions the height fluctuations are governed
by the Baik-Rains distribution. Using the totally asymmetric single step growth model, alias TASEP,
we investigate height fluctuations for a general class of spatially homogeneous random initial conditions.
We prove that for TASEP there is a one-parameter family of limit distributions, labeled by the diffusion coefficient of the initial conditions. The
distributions are defined through a variational formula. We use Monte Carlo simulations to obtain their numerical plots.
Also discussed is the connection to the six-vertex model at its conical point.
\end{abstract}

\section{Introduction}
For stochastic growth models in the Kardar-Parisi-Zhang (KPZ) universality class over a one-dimensional substrate the height fluctuations ``always'' broaden  as $t^{1/3}$. On the other hand the full probability density function depends on the choice of the initial data. As well known, for a flat initial surface, $h(x,t=0) = 0$,  the large $t$ fluctuations of $h(0,t)$ are distributed according to the GOE Tracy-Widom distribution~\cite{TW96,PS00,BR99}. In contrast, if the height profile is macroscopically curved, then GOE has to be replaced by GUE~\cite{TW94,Jo00b,BDJ99,ACQ10,SS10a,FV13,BCG14}.
Such dependence is not so easily inferred directly from the growth dynamics.  But if one switches to the equivalent model of a directed polymer in a space-time random potential, then one end point of the polymer is fixed, while the constraints for the other end point depend on the initial conditions. For example, in the curved case the other end point is fixed and in the flat case the other end point is sampled along a straight line. The size of the polymer free energy fluctuations is robust, but the probability density is sensitive to the particular end point sampling.

In our contribution we study such dependence in case the initial slopes are statistically homogeneous.
Besides flat, the only other  known explicit example refers to initial heights which fluctuate as two-sided Brownian motion. If
as surface growth model we consider the one-dimensional KPZ equation,
\begin{equation}
\partial_t h = \tfrac{1}{2}(\partial_x h)^2 + \tfrac{1}{2}\partial_x^2 h + \xi
\end{equation}
with $\xi(x,t)$ normalized space-time white noise, then the time-stationary initial data are
\begin{equation}
h(x,0) = B(x)
\end{equation}
with $B(x)$ a two-sided Brownian motion. As proved  in~\cite{BCFV14},
\begin{equation}
h(0,t) \simeq -\tfrac1{24}t + (t/2)^{1/3}\xi_\mathrm{BR}
\end{equation}
for large $t$ and the random amplitude $\xi_\mathrm{BR}$ is Baik-Rains distributed~\cite{BR00}. The same limit distribution has been established for other stochastic growth model with stationary initial conditions, see~\cite{FS05a,BR00,Agg16}.

Our interest are spatially homogeneous random initial data, however dropping the assumption of being time-stationary under the growth dynamics.
This is physically a rather natural assumption, since it is realized
by initial slopes being approximately independent. In fact, we will exclusively focus on the one-sided growing single-step model,
\textit{alias} TASEP, with initial  slope (= particle) configuration
$\eta_j(t=0) = \eta_j$ taking values $0,1$. The restriction to spatially homogeneous initial slopes means that $\{\eta_j|j \in \Z\}$ is a stationary stochastic process. We assume the validity of a functional central limit theorem
\begin{equation}\label{eq2a}
\lim_{\ell\to\infty}  \frac{2}{\sqrt{\ell}} \sum_{j = 0}^{[\gamma x \ell]}  \big(\eta_j - \langle\eta_0\rangle\big) = \sigma B(x)
\end{equation}
for some $\sigma \geq 0$. Time-stationary corresponds to $\{\eta_j|j \in \Z\}$ being Bernoulli. In \eqref{eq2a},  $\gamma$ is a scaling constant set such that $\sigma=1$ for Bernoulli. As to be shown, if $\sigma = 0$, then the height fluctuations are GOE Tracy-Widom distributed, thereby confirming a conjecture of~\cite{QR16} in the context of the KPZ equation. In our contribution we will argue, and prove under additional assumptions, that for translation invariant random initial data the universality classes are labeled by $\sigma$. For each $\sigma$ there is a distinct distribution function $F^{(\sigma)}(s)$, as defined through the variational formula
\begin{equation}\label{eq5}
F^{(\sigma)}(s) = \mathbb{P}\Big( \sup_{x\in\R}\{\sqrt{2} \sigma B(x) +\mathcal{A}_2(x) - x^2\} \leq s \Big).
\end{equation}
Here $\mathcal{A}_2(x)$ is the Airy process, independent of the two-sided Brownian motion $B(x)$.
To indicate our main result, let us denote by $\rho$ the expected density of particles and by $\mathsf{j}$ the expected (infinitesimal) current of particles. Time correlations are significant only close to the ray $\{j = \mathsf{j}'(\rho)t\}$. Thus if $\mathsf{j}'(\rho)=0$, the height fluctuations close to the origin, as obtained from $\eta_j(t)$, are governed by $F^{(\sigma)}(s)$ in the large $t$ limit. In case $\mathsf{j}'(\rho)\neq 0$,  $F^{(\sigma)}(s)$ will be observed only when properly centered (see e.g.~\cite{FS05a} in the case $\sigma=1$).

Our proof, to be detailed in Section~\ref{sectVarFormula}, uses the last passage percolation picture. We rely on several non-trivial results obtained only recently, the most important ones being tightness~\cite{CP15b} for the point-to-point process with ending points on horizontal lines, and the one-point slow-decorrelation~\cite{CFP10b}. Finally one also needs to know the convergence of the finite-dimensional distributions~\cite{BP07}. These ingredients can be used to obtain a functional slow-decorrelation result (Theorem~\ref{thmSlowDecorrelation}), see~\cite{CLW16} for the discrete time counterpart. Interestingly, this latter result then implies tightness of the point-to-point process along generic lines, see e.g.\ Corollary~\ref{corTightness}, which is a result not covered by the elegant and soft arguments of~\cite{CP15b}.

Two recent contributions are close to our work. In the context of the KPZ equation, Quastel and Remenik~\cite{QR16} investigate the domains of attraction for general initial conditions. In this model the convergence to the Airy process is not available. Instead the authors have to rely on a tightness argument, which yields weaker results. Corwin, Liu, and Wang ~\cite{CLW16} study the discrete time TASEP, which maps to last passage percolation with geometric random variables. In their main result, Theorem~2.7 of \cite{CLW16}, they consider slightly more general initial conditions and derive a variational formula of the form \eqref{eq5} with $\sqrt{2} \sigma B(x)$ replaced by the limiting initial condition, see also Remark~\ref{rem2.7} below. In our contribution we use the same strategy to control the convergence of the variational process for $x$ only restricted to a finite region, say $[-M,M]$. However, the control on the variational process beyond $[-M,M]$ is simplified and we do not need to appeal to the Brownian Gibbs property of an underlying non-intersecting line ensemble.

As already proved in~\cite{Jo03b}, $F^{(0)}(s)=F_{\rm GOE}(2^{2/3} s)$, with $F_{\rm GOE}$ the GOE Tracy-Widom distribution. As follows from our results,
for $\sigma=1$  we conclude that
\begin{equation}
F^{(1)}(s)=  \mathbb{P}\Big( \sup_{x\in\R}\{\sqrt{2} B(x) +\mathcal{A}_2(x) - x^2\} \leq s \Big)= F_{\rm BR}(s),
\end{equation}
where $F_{\rm BR}$ is the Baik-Rains distribution function~\cite{BR00}, see Corollary~\ref{CorFBR}.
For all other values of $\sigma$ we have to rely on Monte Carlo simulations, which will be  reported in  Section~\ref{sectSimulation}. A convenient choice is to simulate the TASEP dynamics and to independently sample $\eta_j$ as a reversible Markov chain. The average density is taken as $\tfrac{1}{2}$ and varying the only remaining parameter one can achieve \mbox{$0 \leq \sigma < \infty$}. To check the claimed universality we also simulate the diagonal transfer matrix  of the six-vertex model at its stochastic point, see Section~\ref{sectSixVertex}. For the equilibrium Gibbs measure of this model, the Baik-Rains distribution has been proved recently~\cite{Agg16}.

A further explicit  solution of the variational problem \eqref{eq5} results by dropping the Airy term,
\begin{equation}\label{eq5a}
\mathbb{P}\Big( \sup_{x\in\R}\{B(x)  - x^2\} \leq s \Big).
\end{equation}
Formally, one scales $s$ as $\sigma^{4/3}$ and takes  the limit $\sigma \to \infty$, see Appendix~\ref{SectLargeSigma}.
The solution  to \eqref{eq5a} is studied in~\cite{Gro89}. Its probability density vanishes for $s <0$ and decays as a stretched exponential with power $\tfrac{3}{2}$ for $s \to \infty$. This information can be used to obtain  for large $\sigma$ an explicit approximation to $F^{(\sigma)}$.

\bigskip\noindent
{\bf Acknowledgments.} The authors are grateful for discussions with Ivan Corwin and Kazumasa Takeuchi. This work started when two of us visited in early 2016 the Kavli Institute of Theoretical Physics at Santa Barbara. The work of P.L. Ferrari and S. Chhita is supported by the German Research Foundation  as part of the SFB 1060--B04 project.

\section{Convergence to the variational formula}\label{sectVarFormula}
In this section we prove the variational formula (\ref{eq5}) for the totally asymmetric simple exclusion process (TASEP) in continuous time, which is directly related to the last passage percolation (LPP) model with exponential waiting times. Thus we first recall the models and their well known relation.
\subsection{TASEP and LPP} \label{sec:TASEPandLPP}
Let us first recall the relation between TASEP and LPP. A last passage percolation (LPP) model on $\Z^2$ with independent random variables $\{\omega_{i,j},i,j\in\Z\}$ is the following. An \emph{up-right path} $\pi=(\pi(0),\pi(1),\ldots,\pi(n))$ on $\Z^2$ from a point $A$ to a point $E$ is a sequence of points in $\Z^2$ with \mbox{$\pi(k+1)-\pi(k)\in \{(0,1),(1,0)\}$}, with $\pi(0)=A$ and $\pi(n)=E$, and where $n$ is called the length $\ell(\pi)$ of $\pi$. Now, given a set of points $S_A$ and $E$, one defines the last passage time $L_{S_A\to E}$ as
\begin{equation}\label{eq3.2}
L_{S_A\to E}=\max_{\begin{subarray}{c}\pi:A\to E\\A\in S_A\end{subarray}} \sum_{1\leq k\leq \ell(\pi)} \omega_{\pi(k)}.
\end{equation}
Finally, we denote by $\pi^{\rm max}_{S_A\to E}$ any maximizer of the last passage time $L_{S_A\to E}$. For continuous random variables, the maximizer is a.s.\ unique.

TASEP is an interacting particle system on $\Z$ with state space $\Omega=\{0,1\}^\Z$. Here $\eta$ is the occupation variable, which is $1$ at site $j$ if and only if $j$ is occupied by a particle. TASEP has generator $L$ given by~\cite{Li99}
\begin{equation}\label{1.1}
Lf(\eta)=\sum_{j\in\Z}\eta_j(1-\eta_{j+1})\big(f(\eta^{j,j+1})-f(\eta)\big),
\end{equation}
where $f$ are local functions (depending only on finitely many sites) and $\eta^{j,j+1}$ denotes the configuration $\eta$ with the
occupations at sites $j$ and \mbox{$j+1$} interchanged. Notice that for the TASEP the ordering of particles is preserved. That is, if initially one orders from right to left as
\[\ldots < x_2(0) < x_1(0) < 0 \leq x_0(0)< x_{-1}(0)< \cdots,\]
then for all times $t\geq 0$ also $x_{n+1}(t)<x_n(t)$, $n\in\Z$.

Concerning the mapping between TASEP and LPP, the $\omega_{i,j}$ in the LPP is the waiting time of particle $j$ to jump from site $i-j-1$ to site $i-j$.
By definition  $\omega_{i,j}$ are ${\rm exp}(1)$ iid.\ random variables. Let $S_A=\{(u,k)\in\Z^2: u=k+x_k(0), k\in\Z\}$. Then, the relationship between TASEP and LPP is given by
\begin{equation}
\Pb\left(L_{{S_A}\to (m,n)}\leq t\right)=\Pb\left(x_n(t)\geq m-n\right).
\end{equation}
Further, for $m=n$,
\begin{equation}\label{eqLinkLPPtasep}
\Pb\left(L_{{S_A}\to (n,n)}\leq t\right)=\Pb\left(x_n(t)\geq 0\right) = \Pb\left(J(t)\geq n\right),
\end{equation}
where $J(t)$ counts the number of jumps from site $0$ to site $1$ during the time-span $[0,t]$.

To state the variational problem for TASEP, we have to first reformulate the TASEP as
a growth process by introducing the height function $h(j,t)$ which is defined by
\begin{equation}\label{1.11}
h(j,t)=
 \begin{cases}
 2J(t) +\sum^j_{i=1}(1-2\eta_i(t)) & \textrm{for }j\geq 1,\\
 2J(t) & \textrm{for }j=0,\\
 2J(t) -\sum^0_{i=j+1}(1-2\eta_i(t)) & \textrm{for }j\leq -1,
 \end{cases}
\end{equation}
for $j\in\Z$, $t\geq 0$.

The well-studied cases corresponding to $\sigma=0$, resp.\ $\sigma=1$, in (\ref{eq5}), are obtained by choosing periodic initial conditions ($x_j(0)=-2j,j\in\Z$), resp.\ stationary initial conditions~\cite{Lig76} ($\eta_i(0)$, $i\in\Z$, are taken to be i.i.d.\ Bernoulli-$\rho$ random variables, $\rho\in (0,1)$). In this work we consider for simplicity $\rho=1/2$ because in this case the characteristic lines for the associated PDE (Burgers equation) corresponds to the space-time lines with fixed spatial coordinate. Considering other densities does not require any additional technical difficulty, as one  introduces appropriate space-shifts, which only slightly complicates the formulas. Here is the assumption on the initial condition.
\begin{assumpt}\label{assumptBrownianScaling}
The height function $j\mapsto h(j,0)$ weakly (in the uniform topology on bounded sets) converges under a Brownian scaling to a two-sided Brownian motion with drift $0$ and diffusion constant $\sigma^2$. Explicitly,
\begin{equation}\label{eq2.7}
x\mapsto \frac{h([\gamma \ell x],0)}{\sqrt{\ell}} \to \sigma B(x)\quad\textrm{as }\ell\to\infty,
\end{equation}
where $B$ is a standard two-sided Brownian motion. Here $\gamma$ is a model-dependent scaling constant fixed by the property that $\sigma=1$ is the stationary case.
\end{assumpt}

\begin{thm}\label{ThmVarProblem}
Let us define the distribution function
\begin{equation}\label{eq2.8}
F^{(\sigma)}(s):=\Pb\left(\max_{u\in\R} \left[\sqrt{2}\sigma B(u)+ \left({\cal A}_2(u)-u^2\right)\right]\leq s\right).
\end{equation}
Then, under Assumption~\ref{assumptBrownianScaling}, for any $s\in\R$ it holds
\begin{equation}
\lim_{t\to\infty} \Pb\left(h(0,t)\geq t/2-s (t/2)^{1/3}\right)=F^{(\sigma)}(s).
\end{equation}
\end{thm}

For the TASEP with Bernoulli initial measure, $\sigma=1$, the limiting distribution of the rescaled height function has been identified in \cite{FS05a} with the Baik-Rains distribution $F_{\rm BR}$~\cite{BR00}. Thus we arrive at a variational characterization of the Baik-Rains distribution.
\begin{cor}\label{CorFBR}
The Baik-Rains distribution is determined by
\begin{equation}\label{eqFBR}
F_{\rm BR}(s)= \mathbb{P}\Big( \sup_{x\in\R}\{\sqrt{2} B(x) +\mathcal{A}_2(x) - x^2\} \leq s \Big).
\end{equation}
\end{cor}
By KPZ universality such a variational formula appears naturally. However a rigorous mathematical proof has not been accomplished so far. In~\cite{CLW16} the LPP with geometric random variables is considered. If for time-stationary initial conditions one would take the limit to continuous times,  Corollary~\ref{CorFBR} would be recovered. This technical step has not been worked out. The stationary measures of the discrete time TASEP are given by a Markov chain~\cite{Yag86,Sch00}. Thus, in addition, the general results in~\cite{CLW16} would have to be applied to an initial condition which is not a product measure. In this way one would obtain the variational formula. Further, one should recover the Baik-Rains distribution for that model. Such an analysis might be achieved in the setting studied in~\cite{SI04b}. The variational expression (\ref{eqFBR}) is written also in~\cite{QR16}, except that $\mathcal{A}_2$ is not rigorously identified as the Airy$_2$ process, although there is no reason to doubt. The respective missing technical steps are explained in Theorem~1.5 of~\cite{QR16} and subsequent remarks.

\begin{remark}
As opposed to considering the reference point $x=0$, one could choose instead the reference point $x=2\xi (t/2)^{2/3}$, which then leads to
\begin{multline}\label{eq2.9}
\lim_{t\to\infty} \Pb\left(h(2\xi (t/2)^{2/3},t)\geq t/2-s (t/2)^{1/3}\right) \\ =\Pb\left(\max_{u\in\R} \left[\sqrt{2}\sigma B(u)+ \left({\cal A}_2(u)-(u-\xi)^2\right)\right]\leq s\right).
\end{multline}
\end{remark}
For fixed $\sigma$, $\xi\mapsto \max_{u\in\R} \left[\sqrt{2}\sigma B(u)+ \left({\cal A}_2(u)-(u-\xi)^2\right)\right]$ can be viewed as a stochastic process.
The $\sigma = 0$ process is known as the Airy$_1$ process~\cite{Sas07,BFPS06} and $\sigma=1$ process as the Airy$_{\rm stat}$ process~\cite{BFP09}. Thus we arrived at a one-parameter family of interpolating processes.
It is  distinct from the flat to stationary transition discussed in (1.27) of~\cite{QR16} or in (56) of~\cite{CLW16}, which is obtained from mixed initial conditions and by varying the reference point. \bigskip

In our Monte Carlo simulations, see Section \ref{sectSimulation}, the initial data are generated by the following Markov chain:
$\{\eta_{j}(0),j\in\Z\}$ is a stationary Markov chain with transition matrix
\begin{equation}
T=\left(
    \begin{array}{cc}
      \alpha & 1-\alpha \\
      1-\alpha & \alpha \\
    \end{array}
  \right).
\end{equation}
The stationary one-point distribution is $\Pb(\eta_j=0)=\Pb(\eta_j=1)=1/2$.

\begin{lem}\label{lemmaBrownianScaling}
The height function obtained through the stationary Markov chain with transition matrix $T$ satisfies Assumption~\ref{assumptBrownianScaling} with $\gamma=1$ and $\sigma^2=\alpha/(1-\alpha)$.
\end{lem}

This result can be obtained in various ways. One can use the coupling contained in the proof of Theorem~1 of~\cite{ST84} between a Brownian motion and the correlated random walk, which replaces the standard Skorokhod embedding in the proof of Donsker's theorem for the convergence of random walks to Brownian motions. Another option is to consider the two last increment of the height function as the state of the system and then one recovers the Markov property. Using the fact that the height function is then a linear functional of the Markov process one recovers the result by using the Kipnis-Varadhan Theorem~\cite{KV86}.

\begin{remark}
In \cite{CLW16} the discrete time TASEP is considered. The updates are in parallel and $q$ is the probability of not jumping. The analogue variational formula with $\sigma=q^{-1/4}\in(1,\infty)$ is proved in case the initial condition is product Bernoulli, see (53) in~\cite{CLW16}. This initial condition is not stationary~\cite{Yag86,Sch00}. To reach the full range of $\sigma$, one would have to apply the main theorem to a larger family of random initial data.
\end{remark}

\begin{remark}\label{rem2.7}
One could drop the assumption of translation-invariance for the initial height profile, such that the limit on the right side of (\ref{eq2.7}) will be another non-degenerate limit process, say ${\cal R}$. As soon as the condition (\ref{eqCond1}) in the proof is satisfied for ${\cal R}$ instead of $\sigma B$, then Theorem~\ref{ThmVarProblem} will also hold, where of course $\sqrt{2}\sigma B(u)$ is replaced by $\sqrt{2}{\cal R}(u)$.
\end{remark}

\subsection{Proof of Theorem~\ref{ThmVarProblem}}
The key ingredients to prove Theorem~\ref{ThmVarProblem} are (a) functional slow-decorrelation, (b) bounds that ensures that the optimizer stays within a specific region.
Due to Lemma~\ref{lemmaBrownianScaling}, we know that the random line for the discrete model converges weakly to a Brownian motion as a continuous function. Therefore we know that if we restrict the maximization problem to a region of width $M \ell^{2/3}$ around the origin, the random line will typically deviate from the anti-diagonal of order $\sqrt{M} \ell^{1/3}$. We will prove a functional slow-decorrelation which says that the fluctuations of the LPP problem from the anti-diagonal and those lines at distance $\Or(\ell^{\nu})$ away from it with $\nu<1$, do not exceed $\e \ell^{1/3}$ for any $\e>0$, in the $\ell\to\infty$ limit. This gives control of the LPP from the random line, when restricted to a $M\ell^{2/3}$-neighborhood of the origin.  Finally, we show that the probability that the maximizer starts more than $M \ell^{1/3}$ away from the origin tends to zero as $M$ tends to infinity (we get a bound in $M$ which goes to zero and which is uniform in $\ell$).

One main difference between our proof and the one of the case of geometric random variables of~\cite{CLW16}, is that we do not use the Gibbs property of the associated multilayer model to bound the probability of LPP from small segments through point-to-point probabilities, rather we bound this probability directly using the fact that the kernel for the half-line problem is known.
This simplifies the approach for the estimates of the moderate/large deviations made in~\cite{CLW16}. To prove the functional slow-decorrelation theorem we need a tightness result, which for exponential random variables was recently obtained by Cator and Pimentel with an elegant argument~\cite{CP15b}. Interestingly, the functional slow-decorrelation allows to deduce tightness on other space-time cuts, where the soft-argument does not directly apply.

\begin{proof}[Proof of Theorem~\ref{ThmVarProblem}]
The randomness from the initial conditions and the randomness from the dynamics provide two independent sources of randomness in the system.
From the LPP perspective, these two sources become the randomness in the line from which the last passage time is taken and the randomness of the weights $\omega_{i,j}$'s respectively.  As the LPP only allows up-right paths, these must be contained inside the set ${\cal B}_\ell=\{(m,n)\in\R^2\,|\, m,n\leq \tfrac14 \ell\}$.
\emph{All the sets that defined below are subsets of ${\cal B}_\ell$, but we do not write it explicitly to simplify the notations.}
For instance, here we write ${\cal L}=\{(k+x_k(0),k),k\in\Z\}$ instead of $\{(k+x_k(0),k),k\in\Z\}\cap {\cal B}_\ell$.
Set $t=\ell+2^{2/3}\ell^{1/3} s$, which means that $\ell=t-2^{2/3}t^{1/3}s+o(t^{-1/3})$, and (\ref{eqLinkLPPtasep}) gives
\begin{equation}
\begin{aligned}
\Pb(h(0,t)\geq t/2-s (t/2)^{1/3}) &=\Pb(J(t)\geq \tfrac14 t-s 2^{-4/3}t^{1/3}) \\
&= \Pb(L_{{\cal L}\to (\frac14 t-s 2^{-4/3}t^{1/3},\frac14 t-s 2^{-4/3}t^{1/3})}\leq t) \\
&= \Pb(L_{{\cal L}\to (\ell/4,\ell/4)}\leq\ell+2^{2/3}\ell^{1/3} s) =  \Pb(E_{{\cal L},\ell,s}),
\end{aligned}
\end{equation}
where we define the event
\begin{equation}
E_{{\cal L},\ell,s}=\{L_{{\cal L}\to (\ell/4,\ell/4)}\leq \ell+2^{2/3}\ell^{1/3} s\}.
\end{equation}
For any given $M\geq 1$, we define the domains
\begin{equation}\label{eq3.3}
{\cal D}_{M,\ell}=\{(m,n)\in \R^2\,|\, |m+n|\leq (\ell/2)^{2/3}, |m-n|\leq M (\ell/2)^{2/3}\},
\end{equation}
and ${\cal C}_{M,\ell}={\cal C}_{M,\ell}^+ \cup {\cal C}_{M,\ell}^-$, where
\begin{equation}\label{eq3.4}
{\cal C}_{M,\ell}^-=\bigcup_{k=1}^{\e_0 \ell^{1/3}} \{(m,n)\in\R^2\,|\, n\geq k M (\ell/2)^{2/3}, m+n\geq -\tfrac14 k^2 M^2 (\ell/2)^{1/3}\}
\end{equation}
is the region on the right of ${\cal L}^-_{M,\ell}$, and ${\cal C}_{M,\ell}^+$ is the mirror image of ${\cal C}_{M,\ell}^-$ with respect to the axis $\{m=n\}$; see Figure~\ref{FigLPP2}.
\begin{figure}
\begin{center}
\psfrag{L}[lc]{$\color{blue}{{\cal L}}$}
\psfrag{Lp}[rc]{${\cal L}^+_{M,\ell}$}
\psfrag{Lm}[rc]{${\cal L}^-_{M,\ell}$}
\psfrag{D}[rb]{${\cal D}_{M,\ell}$}
\psfrag{C}[cb]{${\cal C}_{M,\ell}$}
\psfrag{m}[bc]{$m$}
\psfrag{n}[lc]{$n$}
\includegraphics[height=5cm]{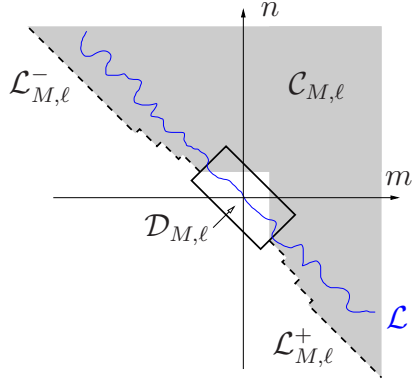}
\caption{Illustration of the main lines and domains used in the proof. The domain ${\cal C}_{M,\ell}$ is the gray one, while ${\cal D}_{M,\ell}$ is the rectangular region. The thick dashed lines are ${\cal L}^\pm_{M,\ell}$.}
\label{FigLPP2}
\end{center}
\end{figure}
For a standard two-sided Brownian motion $B$,
\begin{equation}\label{eqCond1}
\lim_{M\to\infty}\lim_{\ell\to\infty}\Pb(\{(\tfrac12(x+\sigma B(x)),\tfrac12(\sigma B(x)-x)),x\in \R\}\subset {\cal D}_{M,\ell}\cup{\cal C}_{M,\ell})=1.
\end{equation}
This, together with the weak convergence of $\cal L$ to $\sigma B$ (rotated by 45°), see Lemma~\ref{lemmaBrownianScaling} below, implies that
\begin{equation}
\lim_{M\to\infty}\lim_{\ell\to\infty}\Pb({\cal L}\subset {\cal D}_{M,\ell}\cup {\cal C}_{M,\ell}) = 1,
\end{equation}
and consequently
\begin{equation}
\lim_{M\to\infty}\lim_{\ell\to\infty}\Pb(E_{{\cal L},\ell,s})=\Pb(E_{{\cal L},\ell,s} \cap \{{\cal L}\subset {\cal D}_{M,\ell}\cup {\cal C}_{M,\ell}\}).
\end{equation}
For notational simplicity, let us denote by $\Pb_M(\cdot)=\Pb(\cdot \cap \{{\cal L}\subset {\cal D}_{M,\ell}\cup {\cal C}_{M,\ell}\})$.

By definition of the LPP time, $L_{{\cal L}\to (\ell/4,\ell/4)}=\max_{I\in{\cal L}} L_{I\to (\ell/4,\ell/4)}$. The maximum is obtained a.s.\ at a unique point (as the random waiting times have densities). Notice that $E_{{\cal L},\ell,s}$ holds if and only if both $E_{{\cal L}\cap {\cal D}_{M,\ell},\ell,s}$ and $E_{{\cal L}\cap {\cal C}_{M,\ell},\ell,s}$ hold. Theorem~\ref{thmLocalizazion} below implies that for any given $M\geq M_0$ (with $M_0$ a constant)
\begin{equation}\label{eq3.22}
\Pb_M(E_{{\cal L}\cap {\cal C}_{M,\ell},\ell,s}))\geq 1-C e^{-c M}
\end{equation}
uniformly for $\ell\geq \ell_0$, where $C$ and $c$ are constants independent of $\ell$. Therefore
\begin{equation}\label{eq3.24}
\lim_{M\to\infty}\lim_{\ell\to\infty} \Pb_M(E_{{\cal L},\ell,s})= \lim_{M\to\infty}\lim_{\ell\to\infty}\Pb_M(E_{{\cal L}\cap {\cal D}_{M,\ell},\ell,s}).
\end{equation}
For a given $I$, we denote by $I_0$ its orthogonal projection on the line \mbox{$\{(m,n)\in\R^2\,|\, m+n=0\}$}, that is, the anti-diagonal. We now introduce the following coordinates: $I_0=I_0(u)=u (\ell/2)^{2/3} (-1,1)$ and for this $I_0$, we denote by $I=I(u)$ the point on $\cal L$ connected to $I_0(u)$ by a line of slope $1$; see Figure~\ref{FigLPP}.
\begin{figure}
\begin{center}
\psfrag{L}[l]{$\color{blue}{{\cal L}}$}
\psfrag{D}[c]{${\cal D}_{M,\ell}\quad$}
\psfrag{m}[bc]{$m$}
\psfrag{n}[rc]{$n$}
\psfrag{I}[lc]{$I_0(u)$}
\psfrag{Ip}[bl]{$I_+(u)$}
\psfrag{Im}[tc]{$I_-(u)$}
\psfrag{J}[rc]{$I(u)$}
\includegraphics[height=5cm]{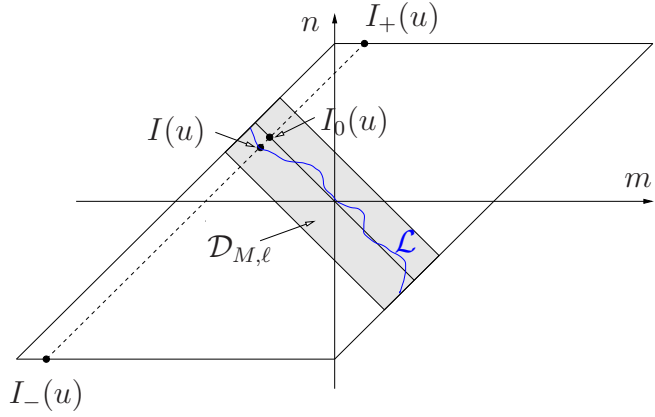}
\caption{LPP with random initial condition ${\cal L}$.}
\label{FigLPP}
\end{center}
\end{figure}
With these notations we can write
\begin{equation}\label{eq3.27}
\begin{aligned}
&E_{{\cal L}\cap {\cal D}_{M,\ell},\ell,s}=\{\max_{|u|\leq M} L_{I(u)\to (\ell/4,\ell/4)}\leq \ell+2^{2/3}\ell^{1/3} s\}\\
&=\{\max_{|u|\leq M} L_{I_0(u)\to (\ell/4,\ell/4)}+ (L_{I(u)\to (\ell/4,\ell/4)}- L_{I_0(u)\to (\ell/4,\ell/4)})\leq \ell+2^{2/3}\ell^{1/3} s\}.
\end{aligned}
\end{equation}

Now we use the functional slow-decorrelation theorem, given in Theorem~\ref{thmSlowDecorrelation} below: for any $\e>0$, the good event
\begin{equation}
G_{{\cal L},M,\ell,\e}=\{\max_{|u|\leq M} |L_{I(u)\to (\ell/4,\ell/4)}-L_{I_0(u)\to (\ell/4,\ell/4)}-4 m(I(u))+4 m(I_0(u))|\leq 2^{2/3} \ell^{1/3}\e\}
\end{equation}
occurs with probability $1$ as $\ell\to\infty$. Therefore for any $\e>0$
\begin{equation}\label{eq3.29}
(\ref{eq3.24})=\lim_{M\to\infty}\lim_{\ell\to\infty} \Pb_M(E_{{\cal L}\cap {\cal D}_{M,\ell},\ell,s}\cap G_{{\cal L},M,\ell,\e}).
\end{equation}
Further,
\begin{equation}\label{eq3.30}
\begin{aligned}
&\Pb_M(E_{{\cal L}\cap {\cal D}_{M,\ell},\ell,s}\cap G_{{\cal L},M,\ell,\e})\\
&= \Pb_M(\{\max_{|u|\leq M} L_{I(u)\to (\ell/4,\ell/4)}\leq \ell+2^{2/3}\ell^{1/3} s\}\cap G_{{\cal L},M,\ell,\e})\\
&\leq \Pb_M\Big(\{\max_{|u|\leq M} \frac{L_{I_0(u)\to (\ell/4,\ell/4)}-\ell+4m(I(u))-4 m(I_0(u))}{2^{2/3} \ell^{1/3}}\leq s+\e\}\cap G_{{\cal L},M,\ell,\e}\Big).
\end{aligned}
\end{equation}
A lower bound to $\Pb_M(E_{{\cal L}\cap {\cal D}_{M,\ell},\ell,s}\cap G_{{\cal L},M,\ell,\e})$ is obtained by replacing $\e$ by $-\e$ in the same way.
By Lemma~\ref{lemmaBrownianScaling}, $\cal L$ converges weakly to a Brownian motion path with diffusivity constant $\sigma^2$ as $\ell\to\infty$, i.e., to $\sigma B$. An elementary calculation gives
\begin{equation}
\lim_{\ell\to\infty}\frac{4m(I(u))-4m(I_0(u))}{2^{2/3} \ell^{1/3}}=\sqrt{2}\sigma B(u).
\end{equation}
Further, $\frac{L_{I_0(u)\to (\ell/4,\ell/4)}-\ell}{2^{2/3} \ell^{1/3}}$ as a stochastic process in $u\in[M,M]$ converges weakly to ${\cal A}_2(u)-u^2$ (with ${\cal A}_2$ being the Airy$_2$ process~\cite{PS02,Jo03}), see Corollary~\ref{corTightness} (alternatively, one could consider $I$ to be the projection to the $m$-axis directly, with the only difference that the difference of the LPP from $I(u)$ and from $I_0(u)$ is now a Brownian motion with a drift, exactly compensating the linear term in (\ref{eq3.19b})). Thus
\begin{equation}
\begin{aligned}
\lim_{M\to\infty}\lim_{\ell\to\infty}(\ref{eq3.30}) &= \lim_{M\to\infty}\Pb\Big(\max_{|u|\leq M}\left[ {\cal A}_2(u)-u^2+\sqrt{2}\sigma B(u)\right]\leq s+\e\Big)\\
&=\Pb\Big(\max_{u\in\R} \left[{\cal A}_2(u)-u^2+\sqrt{2}\sigma B(u)\right]\leq s+\e\Big),
\end{aligned}
\end{equation}
where the last equality is due to the fact that distributions of the positions of the maximum of ${\cal A}_2(u)-u^2/2$ and $\sqrt{2}\sigma B(u)-u^2/2$ are tight (for the Airy$_2$ case, see Proposition 4.4 of~\cite{CH11} or Proposition 2.13 in~\cite{CLW16}, while for the Brownian motion it is obvious as we can control it with a Brownian motion with a drift away from the origin; see~\cite{Gro85,Gro89,Gro10} for explicit formulas). This, together with the analogous lower bound holds for any $\e>0$. Thus we have shown the convergence in distribution in Theorem~\ref{ThmVarProblem}. The limiting distribution function in (\ref{eq2.8}) is continuous as the Airy$_2$ process is locally Brownian (use inequalities (103)-(104) and the localizing bound (88) in~\cite{CLW16}).
\end{proof}

In the next theorem we show that the LPP from ${\cal C}_{M,\ell}$ is typically lower than the one from the origin, from which one can prove that the maximizer is localized to occur in ${\cal D}_{M,\ell}$. The probability that this does not occur goes to zero as $M\to\infty$ as expected, since ${\cal C}_{M,\ell}$ stays on the right of the constant limit shape.
\begin{thm}\label{thmLocalizazion}
Let ${\cal C}_{M,\ell}$ defined just after (\ref{eq3.3}) and let ${\cal L}_{M,\ell}=\partial {\cal C}_{M,\ell}$ be the boundary of it. For any fixed $s\in\R$, there is a $M_0>0$ such that for all given $M>M_0$, there exists a $\ell_0\in (0,\infty)$ such that
\begin{equation}
\Pb(L_{{\cal L}_{M,\ell}\to(\ell/4,\ell/4)}< \ell+2^{2/3}\ell^{1/3}s)> 1-C e^{-c M}
\end{equation}
uniformly for $\ell\geq \ell_0$. The constants $C$ and $c$ are $\ell$-independent.
\end{thm}
\begin{proof}
Denote by
\begin{equation}
{\cal L}_{kM,\ell}=\partial \{(m,n)\in\R^2\,|\, n > k M (\ell/2)^{2/3}, m+n > -\tfrac14 k^2 M^2 (\ell/2)^{1/3}\}.
\end{equation}
Symmetry and union bounds give
\begin{equation}
\Pb(L_{{\cal L}_{M,\ell}\to(\ell/4,\ell/4)}> \ell+2^{2/3}\ell^{1/3}s)\leq 2 \sum_{k=1}^{\e_0\ell^{1/3}} \Pb(L_{{\cal L}_{kM,\ell}\to(\ell/4,\ell/4)} > \ell+2^{2/3}\ell^{1/3}s). \label{eq:LPPsum}
\end{equation}

In order to study the right side of the above equation,  recall the mapping between LPP and TASEP as given in~\eqref{eqLinkLPPtasep}.  Since ${\cal L}_{kM,\ell}\to (\ell/4,\ell/4)$ is a half-line  with slope $-1$ taking values in the second quadrant, this LPP corresponds to TASEP with an initial condition in which particles occupy every second position below a threshold and there are no particles above the threshold.
The distribution function of this TASEP is given by a Fredholm determinant with an explicit kernel derived in~\cite{BFS07}. Namely, from~\cite[Proposition 3]{BFS07}, we have the formula
\begin{equation}
\Pb (x_n (t) \geq a) = 1- \det(\Id - \chi_a K_{n,t} \chi_a)_{\ell^2 (\Z)} \label{eq:Fredholm}
\end{equation}
where $\chi_a(x)=\Id_{[x<a]}$ and
\begin{equation}
K_{n,t} (y_1,y_2) = \frac{1}{(2 \pi \I )^2} \oint_{\Gamma_0} \hspace{-0.5em}\mathrm{d}v \oint_{\Gamma_{-1,-1-v}} \hspace{-2.5em}\mathrm{d}u \frac{e^{-vt} (1+v)^{y_2+n} }{v^n}  \frac{e^{ut} u^n}{(1+u)^{y_1+n+1}} \frac{1+2v}{(u-v)(1+u+v)} \label{eq:kernel}
\end{equation}
where $-\Gamma_0-1$ is contained inside $\Gamma_{-1,-1-v}$. Note that this kernel is equal to zero if $y_1 \leq -2n-1$. Indeed, in this case the only pole for $u$ is at $u=-1-v$ and the integration over this simple poles leads to an integrand in $v$ which does not have a pole at $0$ anymore. To find the specific formula for this theorem, notice that by a translation of the LPP, the  summand on the right side of~\eqref{eq:LPPsum}
gives
\begin{equation}
\begin{aligned}
&\Pb(L_{{\cal L}_{kM,\ell}\to(\ell/4,\ell/4)}> \ell+2^{2/3}\ell^{1/3}s)\\&= \Pb(L_{{\cal L}_{0,0}\to(\ell/4+(k M)(\ell/2)^{2/3} +\frac{1}{4} (kM)^2(\ell/2)^{1/3} ,\ell/4-(k M)(\ell/2)^{2/3})}>  \ell+2^{2/3}\ell^{1/3}s).\label{eq:LPPtranslation}
\end{aligned}
\end{equation}
Using the TASEP and LPP mapping, ~\eqref{eqLinkLPPtasep}, we have
\begin{equation}
\begin{aligned} \label{eq:lpptoTASEP}
&\Pb(L_{{\cal L}_{0,0}\to(\ell/4+(k M)(\ell/2)^{2/3} + \frac{1}{4} (kM)^2(\ell/2)^{1/3} ,\ell/4-(k M)(\ell/2)^{2/3})}>  \ell+2^{2/3}\ell^{1/3}s) \\
&=\Pb(x_{\ell/4 -(kM)(\ell/2)^{2/3} }(\ell+2^{2/3} \ell^{1/3} s) < 2 (kM)(\ell/2)^{2/3} + \tfrac{1}{4} (k^2 M^2)(\ell/2)^{1/3}).
\end{aligned}
\end{equation}
The above equation indicates the  scalings
\begin{equation}
\begin{aligned} \label{eq:scaling}
t&= \ell+2^{2/3} \ell^{1/3} s\\
n&= \ell/4- (kM)(\ell/2)^{2/3} \\
y_i&= 2 (k M) (\ell/2)^{2/3} +\tfrac{1}{4}(k^2 M^2)(\ell/2)^{1/3} -2 (\ell/2)^{1/3} \xi_i.
\end{aligned}
\end{equation}
with $\xi_1,\xi_2 \in [0, \infty)$, where the latter condition is from the indicator function in~\eqref{eq:Fredholm}. We define the rescaled and conjugated kernel to be
\begin{equation}
K^{\mathrm{resc}} (\xi_1,\xi_2)= 2^{y_2-y_1+1} \left(\ell/2\right)^{1/3} K_{n,t} (y_1,y_2)
\end{equation}
with $n$ and $t$ given in~\eqref{eq:scaling}.

To conclude the proof of the theorem, we need bounds on $K^{\mathrm{resc}}$, which are given in Lemma~\ref{lem:moderatedeviations} below.
 Set $a= 2(kM)(\ell/2)^{2/3}+\tfrac14(k^2M^2)(\ell/2)^{1/3}$.  We expand the Fredholm determinant associated to the right side of~\eqref{eq:lpptoTASEP} which gives
\begin{equation}
\sum_{m=0}^\infty \frac{(-1)^m}{m!} \sum_{z_1=-\infty}^{\infty} \dots\sum_{z_m=-\infty}^{\infty} \prod_{i=1}^m \chi_a(z_i) \det \left[ K_{t,n}(z_i,z_j) \right]_{1 \leq i,j \leq m}.
\end{equation}
Now we consider the change of variable $z_i=a-2\xi_i (\ell/2)^{1/3}$ and extend the rescaled kernel as piecewise constant so that we can write the Riemann sums as integrals. The cut-off by $\chi_a(z_i)$ corresponds to have $\xi_i\geq 0$. Thus we get
\begin{equation} \label{eq:fredholmexp}
\sum_{m=0}^\infty \frac{(-1)^m}{m!} \int_0^{\infty} \mathrm{d}\xi_1 \dots \int_0^{\infty} \mathrm{d}\xi_m
\det \left[ K^{\mathrm{resc}}(\xi_i,\xi_j) \right]_{1 \leq i,j \leq m}.
\end{equation}

Let us choose $M$ large enough (such that $|s|<M$), then from Lemma~\ref{lem:moderatedeviations}, we factor  out  $e^{-c(k M)^3} e^{-\xi_i}$ from the above determinant. Using Hadamard's bound in the standard way we can conclude that~\eqref{eq:fredholmexp} is bounded above by $Ce^{-ck^3 M^3}<Ce^{-ckM}$.  Therefore, the right side of~\eqref{eq:LPPsum} is bounded above by
\begin{equation}
\sum_{k=1}^{\eps_0 \ell^{1/3}}C e^{-ckM} \leq C e^{-cM}(1-e^{-cM})^{-1}\leq 2C e^{-cM}
\end{equation}
for $M$ large enough. Since the estimate found in  Lemma~\ref{lem:moderatedeviations} is uniform, the above bound is uniform in $\ell$.
\end{proof}

\begin{lem} \label{lem:moderatedeviations}
Let $M\in (0,\infty)$ be given and $|s|\leq M$. Then, there exists an $\ell_0\in(0,\infty)$ such that for all $\ell\geq \ell_0$ the estimate
\begin{equation}
|K^{\mathrm{resc}} (\xi_1,\xi_2)|\leq\frac{C}{(k M)^2} e^{-c_1(k M)^3-c_2(k M)(\xi_2+s)}
\end{equation}
holds for all $k \in \{1, \dots, \eps_0 \ell^{1/3}\}$ and for all $\xi_1,\xi_2 \geq 0$, where $C,c_1,c_2>0$ are constants which do not depend on $\ell$.
\end{lem}

\begin{proof}[Proof of Lemma~\ref{lem:moderatedeviations}]
For convenience, set $\varepsilon=\varepsilon(k,M,\ell)= kM(\ell/2)^{-1/3}$.

The proof of this result is via a steepest descent analysis. We remark that we do not deform the contours of integration to pass through the critical points. It is more convenient to deform the contours to be relatively close to the critical point giving a less optimal estimate, but this is sufficient for our purposes.

Define
\begin{equation}\label{eq:g0}
g_0 (v) = -4 v+ (1+ 2\varepsilon +\varepsilon^2/2) \log 2(1+v) -(1-2 \varepsilon )\log v.
\end{equation}
Then, with the scalings given in~\eqref{eq:scaling}, we have
\begin{equation}\label{eq:g0expand}
\log \left(\frac{e^{-vt} (2(1+v))^{y_2+n} }{v^n} \right) = \frac{\ell}{4} g_0(v) -2\left( \frac{\ell}{2}\right)^{1/3} \xi_2 \log (2(1+v)) -
2\left( \frac{\ell}{2}\right)^{1/3} sv.
\end{equation}
We have a similar expression for $\log (e^{-ut}u^n/ (2(1+u))^{y_1+n}) $.

Define $C_0=\{v:|v|=1/2-\eps\}$, $C_1=\{u:|u+1|=1/2\}$ and set $\Gamma_0=C_0$, $\Gamma_{-1,-1-v}=C_1$. Further, set $C_0^{\delta}=\{v\in C_0:|v+1/2-\eps|<\delta \}$ and $C_1^{\delta}=\{u\in C_1:|u+1/2|<\delta\}$.  Then, we can write $C_0\cup C_1= C_0^\delta\cup C_1^{\delta}+\Sigma$, where $\Sigma$ is the remaining contour.
Notice that for $v\in \Sigma$, $\mathrm{Re}( g_0(v)-g_0(-1/2+\eps)) \leq -c_3$ for some $c_3=c_3(\delta)$ because $\mathrm{Re}(v)$ increases when moving away from $-1/2+\eps$ along the contour $C_0$ and the same is true for $\log|1+v|$.   By a similar reasoning, we have $\mathrm{Re} (g_0(u)-g_0(-1/2)) \geq c_3$.
We also have  for $u \in C_1$ and $v \in C_0$,
\begin{equation}\label{eq:prefactorbound}
\left|\frac{1+2v}{(u-v)(1+u+v)}\right| \leq \frac{C}{\eps}.
\end{equation}
Using the above bound and observations, we conclude that
\begin{equation}
\begin{aligned}
&\left(\frac{\ell}{2}\right)^{1/3}\left| \frac{4}{(2 \pi \I )^2} \iint_{\Sigma}\mathrm{d}u \, \mathrm{d} v  \frac{e^{-vt} (2(1+v))^{y_2+n} }{v^n}  \frac{e^{ut} u^n}{(2(1+u))^{y_1+n+1}} \frac{1+2v}{(u-v)(1+u+v)} \right|  \\
&\leq C \frac{\ell^{1/3}}{\eps} e^{-c_3 \ell/2}
\end{aligned}
\end{equation}
for $\ell$ large enough. This is because the smaller terms in the exponential in $\ell$ are bounded by $\ell c_3$, which dominates the $\ell^{1/3}$ terms.

For the integral over $C_k^\delta$'s, using $|u+1|=1/2$, we have
\begin{equation}\label{eq:firstbound}
\begin{aligned}
&\left|\frac{ 4\left(\frac{\ell}{2} \right)^{1/3}}{(2 \pi \I )^2} \int_{C_1^{\delta}}  \mathrm{d} u \int_{C_0^{\delta}} \mathrm{d}v \frac{e^{-vt}( 2(1+v))^{y_2+n} }{v^n}  \frac{e^{ut} u^n}{(2(1+u))^{y_1+n+1}} \frac{1+2v}{(u-v)(1+u+v)} \right| \\
&\leq \frac{\left(\frac{\ell}{2} \right)^{1/3}}{\pi^2}
 \int_{C_1^{\delta}} \!\!\! \mathrm{d} u \int_{C_0^{\delta}} \!\!\! \mathrm{d}v \left| e^{\frac{\ell}{4}(g_0(v)-g_0(u))}e^{-2\left( \frac{ \ell}{2}\right)^{1/3}(\xi_2 \log(2 (1+v))+s v-su)} \right|\left|\frac{1+2v}{(u-v)(1+u+v)} \right|.
\end{aligned}
\end{equation}

Introduce $v=-\frac{1}{2}+\eps+V$, $u=-\frac{1}{2}+U$, $\gamma_0^{\delta}=C_0^{\delta}-1/2+\eps$ and $\gamma_1^{\delta}=C_1^{\delta}-1/2$. A Taylor series expansion gives
\begin{equation}
\begin{aligned}
g_0(u)&= g_0(-\tfrac12)+ g_0'(-\tfrac12) U+\tfrac12 g_0''(-\tfrac12) U^2+O(U^3) \\
 &=g_0(-\tfrac12)+\eps^2 U-\eps(8+\eps) U^2+O(U^3)
\end{aligned}
\end{equation}
and
 \begin{equation}
\begin{aligned}
g_0(-\tfrac12+\eps)&= g_0(-\tfrac12+\eps) +g_0'(-\tfrac12+\eps)V +\tfrac12 g_0''(-\tfrac12+\eps) V^2+O(V^3) \\
&=g_0(-\tfrac12+\eps)+\frac{\eps^2}{1+2 \eps} V + \frac{2 \varepsilon  (\varepsilon  (2 \varepsilon +15)+8)}{(1-2 \varepsilon ) (2 \varepsilon +1)^2} V^2+O(V^3).
\end{aligned}
\end{equation}
Set
\begin{equation}
P= \exp\bigg[ \frac{\ell}{4} \Re\left(g_0(-\tfrac12+\eps)-g_0(-\tfrac12)\right) -2\left( \frac{ \ell}{2}\right)^{1/3}(  \xi_2 \log (1+2\eps)+ \eps s) \bigg]
\end{equation}
which is computed later.
Using the above  change of variables and expansions for the right side of~\eqref{eq:firstbound} gives
\begin{equation}
\begin{aligned} \label{eq:firstboundmod1}
 &\frac{ 2^2\left(\frac{\ell}{2} \right)^{1/3}P}{(2\pi)^2}
 \int_{\gamma_1^{\delta}}  \mathrm{d} U \int_{\gamma_0^{\delta}} \mathrm{d}V
\left| \frac{2\epsilon+2V}{(U+\eps+V)(\eps+V-U)}\right|\bigg| e^{O(\ell V^3,\ell U^3, \ell^{1/3}(\xi_2+s) V^2)}  \\
&\times \exp\left[ \frac{\ell}{4}\left(
\frac{\eps^2}{1+2 \eps} V + \frac{2 \varepsilon  (\varepsilon  (2 \varepsilon +15)+8)}{(1-2 \varepsilon ) (2 \varepsilon +1)^2} V^2
- \eps^2 U+\eps(8+\eps) U^2\right) \right] \bigg|.
\end{aligned}
\end{equation}
To the above equation, we make a further change of variables
\begin{equation}
U=\frac{w}{\ell^{1/2}}\frac{\sqrt{2}}{\sqrt{\eps(8+\eps)} } \hspace{5mm} \mbox{and}
\hspace{5mm}
V=\frac{z}{\ell^{1/2}} \frac{\sqrt{1-2 \varepsilon} (2 \varepsilon +1)}{\sqrt{\varepsilon  (\varepsilon  (2 \varepsilon +15)+8)}}
\end{equation}
and let $\tilde{\gamma}_1^{\delta}$ and $\tilde{\gamma}_0^{\delta}$ be the images of $\gamma_1^\delta$ and $\gamma_0^\delta$ under these change of variables.   These change of variables applied to \eqref{eq:firstboundmod1} imply that \eqref{eq:firstbound} is bounded above by
\begin{equation} \label{eq:firstboundmod2}
\begin{aligned}
&\frac{C P \ell^{1/3} }{\eps \ell} \int_{\tilde{\gamma}_1^{\delta}} \mathrm{d}w \int_{\tilde{\gamma}_0^{\delta}} \mathrm{d} z \left| \frac{ 2\eps+O(z\ell^{-1/2})}{\eps^2+O(\eps z \ell^{-1/2},\eps z \ell^{-1/2})} \right|
\left|e^{O(z^3 \ell^{-1/2},w^3 \ell^{-1/2}, \ell^{-2/3}(\xi_2+s)z^2)}\right|\\
&\times \left|e^{\frac{1}{2}(w^2+z^2)}\right|
\bigg|\exp \bigg[ \eps^2\frac{z}{\ell^{1/2}}\frac{\sqrt{ 1-2 \varepsilon }}{\sqrt{\varepsilon  (\varepsilon  (2 \varepsilon +15)+8)}}
-\eps^2\frac{w\sqrt{2}}{\ell^{1/2}\sqrt{\eps(8+\eps)} }\bigg] \bigg|.
\end{aligned}
\end{equation}
A good approximation of the contour $\tilde{\gamma}_1^{\delta}$ is an interval on the imaginary axis. This means that the exponential term involving $z$ is highly oscillatory but has an absolute value equal to 1.
 Choose $\delta$ small enough so that $O(\ell^{-1/2} z^3)\ll z^2$, $O(\ell^{-1/2} z)\ll 1$ (and also $O(\ell^{-1/2}z)  \ll O(\eps)$). Furthermore, choose $\delta$ small enough suc that $M\delta^2\ll 1$.  This means that the exponential terms in the above integral with respect to $z$ are bounded above by $|\exp(\chi_0\frac{z^2}{2})|$ with  $\chi_0>0$. Here, $\chi_0$ can be made close to 1, by choosing $\delta$ small enough. A similar statement can be made for the integral with respect to $w$.  These observations to~\eqref{eq:firstboundmod2} imply that ~\eqref{eq:firstbound} is bounded by
\begin{equation}
\frac{C P \ell^{1/3} }{\eps^2 \ell} \int_{\tilde{\gamma}_1^{\delta}} \mathrm{d}z \int_{\tilde{\gamma}_0^{\delta}} \mathrm{d} w  \Big|e^{\frac{1}{2}(w^2+z^2)}\Big|.
\end{equation}
The contours $\tilde{\gamma}_0^{\delta}$ and $\tilde{\gamma}_0^{\delta}$ can be extended to the whole imaginary axis with no significant error, resulting in a computable Gaussian integral which is finite. Using $\eps= k M(\ell/2)^{-1/3}$ we obtain
\begin{equation}\label{eq:lastbound}
|\eqref{eq:firstbound}|\leq \frac{CP}{(k M)^2}.
\end{equation}

It remains to compute $P$.  We have that
\begin{equation}
\Re (g_0(-\tfrac12+\eps)-g_0(-\tfrac12))=-4\eps +(1+2\eps+\tfrac12\eps^2) \log(1+2\eps)-(1-2\eps) \log (1-2\eps).
\end{equation}
Provided that $\eps_0$ is small enough, we apply a Taylor series expansion to find that
\begin{equation}
\Re (g_0(-\tfrac12+\eps)-g_0(-\tfrac12))=-\frac{5}{3} \eps^3+O(\eps^4),
\end{equation}
which means
\begin{equation}
\frac{\ell}{4}\Re (g_0(-\tfrac12+\eps)-g_0(-\tfrac12))=-c_1 (k M)^3+O(\eps_0\ell^{-1/3}).
\end{equation}
We also have
\begin{equation}
2 \left(\frac{\ell}{2} \right)^{1/3} (\xi_2 \log (1+2\eps)+\eps s) = c_2 k M (\xi_2(1+O(\eps_0  \ell^{-1/3}))+s),
\end{equation}
which means that $P \leq e^{-c_1 (k M)^3-c_2 (k M)(\xi_2+s)}$ for an appropriate choice of positive constants $c_1$ and $c_2$. Substituting this bound for $P$ into~\eqref{eq:lastbound} gives the result.
\end{proof}

Below we prove the missing ingredients in the proof of Theorem~\ref{ThmVarProblem}, in particular the functional slow-decorrelation result (Theorem~\ref{thmSlowDecorrelation}). This result is the analogue of Theorem~2.15 of~\cite{CLW16} for the case of geometric random variables. In the proof one needs tightness (the analogue of Lemma~5.3 of~\cite{Jo03b} for the geometric case). Fortunately, for the case of exponential random variables, this result was obtained with soft arguments in~\cite{CP15b}.
\begin{prop}[Compare with Theorem~4 of~\cite{CP15b}]\label{PropTightnessCator}
Consider the LPP from the horizontal line $\{(m,n)\in\Z^2\,|\,n=0\}$, and rescale it as
\begin{equation}\label{eq3.19b}
L^{\rm res,h}_\ell(u):=\frac{L_{(-2 u (\ell/2)^{2/3},0)\to(\ell/4,\ell/4)}-(\ell +4 u (\ell/2)^{2/3}-2^{2/3}u^2\ell^{1/3})}{2^{2/3}\ell^{1/3}},
\end{equation}
for $u\in\R$ with $2 u (\ell/2)^{2/3}\in\Z$ and linearly interpolate for the other values of $u$. Then, for any given $M>0$,
the collection $\{L^{\rm res,h}_\ell\}$ is tight in the space of continuous functions of $[-M,M]$, ${\cal C}([-M,M])$. Furthermore, $L^{\rm res,h}_\ell$ as a stochastic process in $u\in[M,M]$ converges weakly to the Airy$_2$ process.
\end{prop}
\begin{proof}
Tightness of $L^{\rm res,h}_\ell$ was already shown in Theorem~4 of~\cite{CP15b} with a soft argument that uses a comparison with the stationary case. The basic idea is to compare, for fixed $n$, the increments of $m\mapsto L_{(0,0)\to (m,n)}$ with the same increment for stationary initial conditions with densities $\rho^\pm=1/2\pm\kappa n^{-1/3}$. On some events with high probability (going to $1$ as $\kappa\to\infty$) one has an explicit upper (resp.\ lower) bounds of $L_{(0,0)\to (m+x,n)}-L_{(0,0)\to (m,n)}$ in terms of increments of the stationary case with density $\rho^+$ (resp.\ $\rho^-$). The increments of the latter are easy to control since they are sums of independent $\exp(\rho^\pm)$ random variables. In~\cite{CP15b} paper, the statement is written in the space of c\`adl\`ag functions. The reason is that the detailed proof is given only for a related model (the Hammersley process / LPP in Poisson points) in which c\`adl\`ag is the natural way of describing. However by inspecting the bounds used in the proof, one sees that these are strong enough to get tightness in the space of continuous functions as well.

The convergence of finite dimensional distributions for this exponential waiting times and along the horizontal line is a special case of~\cite{BP07} (or can be obtained from the finite-dimensional distributions along other lines using slow-decorrelation~\cite{CFP10b}; see~\cite{BFP09} for an application of this technique). These two ingredients imply weak convergence~\cite{Bil68}.
\end{proof}
\begin{rem}
For geometric random variables this statement was shown in Theorem~1.2 in~\cite{Jo03b}, which is obtained by controlling the modulus of continuity (Lemma~5.3 in~\cite{Jo03b}) together with the convergence of the finite-dimensional distributions.
\end{rem}

We are ready to prove the functional slow-decorrelation theorem. This is the analogue of Theorem~2.15 of~\cite{CLW16}, which holds for geometric random variables, (slightly) specialized to our setting.
\begin{thm}\label{thmSlowDecorrelation}
Let $\cal L$ be any down-right path traversing ${\cal D}_{M,\ell}$ between the two sides of ${\cal D}_{M,\ell}$ with $|m-n|=2 M (\ell/2)^{2/3}$ (as in Figure~\ref{FigLPP}). For $I_0(u)=u (\ell/2)^{2/3} (1,-1)$, let $I(u)$ be the point in $\cal L$ whose orthogonal projection along $(1,1)$ is $I_0(u)$. Denote by $\mu((\ell_1,\ell_2))=(\sqrt{\ell_1}+\sqrt{\ell_2})^2$ to be the limit shape approximation of $L_{(0,0)\to (\ell_1,\ell_2)}$. Then consider the rescaled LPP processes from the anti-diagonal, $L^{\rm res,ad}$, and from the line $\cal L$, $L^{\rm res,{\cal L}}$ (with linear interpolation as in Proposition~\ref{PropTightnessCator}).
Define
\begin{equation}\label{eq3.20}
\begin{aligned}
L^{\rm res,ad}_\ell(u)&:=\frac{L_{I_0(u)\to(\ell/4,\ell/4)}-\mu((\ell/4,\ell/4)-I_0(u))}{2^{2/3}\ell^{1/3}},\\
L^{\rm res,{\cal L}}_\ell(u)&:=\frac{L_{I(u)\to(\ell/4,\ell/4)}-\mu((\ell/4,\ell/4)-I(u))}{2^{2/3}\ell^{1/3}}.
\end{aligned}
\end{equation}
Then, $L^{\rm res,ad}_\ell-L^{\rm res,{\cal L}}_\ell$ converges in probability to $0$ in ${\cal C}([-M,M])$ as $\ell\to\infty$. Explicitly, there are $\e,\delta>0$ and a $\ell_0$, such that for all $\ell\geq \ell_0$,
\begin{equation}\label{eq3.21}
\Pb\Big(\max_{|u|\leq M} |L^{\rm res,ad}_\ell(u)-L^{\rm res,{\cal L}}_\ell(u)|\geq \delta\Big)<\e.
\end{equation}
\end{thm}
In the proof, we use known results for the point-to-point LPP with exponential random variables, which we recall here.
\begin{prop}\label{PropBounds}
For $\eta\in(0,\infty)$ define $\mu=(\sqrt{\eta \ell}+\sqrt{\ell})^2$, $\sigma=\eta^{-1/6}(1+\sqrt{\eta})^{4/3}$, and the rescaled random variable
\begin{equation}
L^{\rm res}_\ell:=\frac{L_{(0,0)\to(\eta\ell,\ell)}-\mu}{\sigma\ell^{1/3}}.
\end{equation}
(a) Limit law
\begin{equation}
\lim_{\ell\to\infty} \Pb(L^{\rm res}_\ell\leq s) = F_{\rm GUE}(s),
\end{equation}
with $F_{\rm GUE}$ the GUE Tracy-Widom distribution function.\\
(b) Bound on upper tail: there exist constants $s_0,\ell_0,C,c$ such that
\begin{equation}
\Pb(L^{\rm res}_\ell\geq s)\leq C e^{-c s}
\end{equation}
for all $\ell\geq \ell_0$ and $s\geq s_0$.\\
(c) Bound on lower tail: there exists constants $s_0,\ell_0,C,c$ such that
\begin{equation}
\Pb(L^{\rm res}_\ell\leq s)\leq C e^{-c |s|^{3/2}}
\end{equation}
for all $\ell\geq \ell_0$ and $s\leq -s_0$.
\end{prop}
(a) was proven in Theorem~1.6 of\cite{Jo00b}. Using the relation with the Laguerre ensemble of random matrices (Proposition~6.1 of~\cite{BBP06}), or to TASEP described above, the distribution is given by a Fredholm determinant. An exponential decay of its kernel leads directly to (b). See e.g.\ Proposition~4.2 of~\cite{FN13} or Lemma~1 of~\cite{BFP12} for an explicit statement. (c) was proven in~\cite{BFP12} (Propositon~3 together with (56)). In the present language it is reported in Proposition~4.3 of~\cite{FN13} as well.
\begin{proof}[Proof of Theorem~\ref{thmSlowDecorrelation}]
The proof is similar to the one of Theorem~2.15 of~\cite{CLW16}, except that this time tightness is known on a horizontal slice of $\Z^2$ (see Proposition~\ref{PropTightnessCator}). To prove (\ref{eq3.21}) it is enough to prove it for $L^{\rm res,h}_\ell$ instead of $L^{\rm res,ad}_\ell$. Indeed, the anti-diagonal satisfies the requirements for ${\cal L}$ and one uses the triangle inequality
\begin{equation}
|L^{\rm res,ad}_\ell(u)-L^{\rm res,{\cal L}}_\ell(u)| \leq |L^{\rm res,h}_\ell(u)-L^{\rm res,ad}_\ell(u)| + |L^{\rm res,h}_\ell(u)-L^{\rm res,{\cal L}}_\ell(u)|.
\end{equation}
Thus we need to show
\begin{equation}\label{eq3.21b}
\Pb\Big(\max_{|u|\leq M} |L^{\rm res,h}_\ell(u)-L^{\rm res,{\cal L}}_\ell(u)|\geq \delta\Big)<\e.
\end{equation}

Let us consider two horizontal lines that `enclose' the set ${\cal D}_{M,\ell}$, see Figure~\ref{FigLPP}.
We choose the lines parameterized by
\begin{equation}
I_\pm(u)=-2u(\ell/2)^{2/3} (1,0)\pm 2M (\ell/2)^{2/3} (1,1),\quad |u|\leq M,
\end{equation}
and define $L^{\rm res,\pm}_\ell(u)$ as in (\ref{eq3.20}) with $I_0(u)$ replaced by $I_\pm(u)$. To prove (\ref{eq3.21b}) we use the inequalities
\begin{equation}\label{eq3.31}
\begin{aligned}
\Pb\Big(\max_{|u|\leq M} L^{\rm res,h}_\ell(u)-L^{\rm res,{\cal L}}_\ell(u)\geq \delta\Big) &\leq
\Pb\Big(\max_{|u|\leq M} L^{\rm res,h}_\ell(u)-L^{\rm res,+}_\ell(u)\geq \delta/2\Big)\\
&+\Pb\Big(\max_{|u|\leq M} L^{\rm res,+}_\ell(u)-L^{\rm res,{\cal L}}_\ell(u)\geq \delta/2\Big)
\end{aligned}
\end{equation}
and similarly
\begin{equation}\label{eq3.32}
\begin{aligned}
\Pb\Big(\max_{|u|\leq M} L^{\rm res,{\cal L}}_\ell(u)-L^{\rm res,h}_\ell(u)\geq \delta\Big) &\leq
\Pb\Big(\max_{|u|\leq M} L^{\rm res,{\cal L}}_\ell(u)-L^{\rm res,-}_\ell(u)\geq \delta/2\Big)\\
&+\Pb\Big(\max_{|u|\leq M} L^{\rm res,-}_\ell(u)-L^{\rm res,h}_\ell(u)\geq \delta/2\Big).
\end{aligned}
\end{equation}

Below we consider only (\ref{eq3.31}) as (\ref{eq3.32}) is threaded in a similar way. First we prove the following: for any $\e,\delta>0$, there exists a $\ell_0$ (depending on $M$ only) such that for all $\ell\geq\ell_0$,
\begin{equation}\label{eq3.23}
\Pb\Big(\max_{|u|\leq M} |L^{\rm res,+}_\ell(u)-L^{\rm res,h}_\ell(u)|\geq \delta/2\Big)<\e/2.
\end{equation}
First of all, notice that tightness of $L^{\rm res,h}_\ell$ implies also tightness of $L^{\rm res,\pm}_\ell$ as they are essentially the same process:
\begin{equation}
L^{\rm res,+}_\ell(u)=(1+4M(\ell/2)^{-1/3})^{-1/3} L^{\rm res,h}_{\ell-8M(\ell/2)^{2/3}}(u+\Or(\ell^{-1/3})).
\end{equation}
Thus, for $M,\e,\delta$ given above, there exists constants $\delta_1,\ell_1$ such that for all $\ell\geq \ell_1$,
\begin{equation}\label{eq3.25}
\Pb\Big(\max_{|u_1|,|u_2|\leq M, |u_1-u_2|\leq \delta_1} |L^{\rm res,*}_\ell(u_1)-L^{\rm res,*}_\ell(u_2)|\geq \delta/6\Big)<\e/6,\quad *\in\{+,{\rm h}\}.
\end{equation}
Let us define $u_j= j \delta_1$ with $j\in\Z$ satisfying $u_j\in[-M,M]$. Since the number of $j$'s is finite, from the usual slow-decorrelation result (Theorem~2.1 of~\cite{CFP10b}), there exists some constant $\ell_2$ (depending on $\e,\delta,\delta_1$) such that for all $\ell\geq \ell_2$,
\begin{equation}\label{eq3.26}
\Pb\Big(\max_{j: |u_j|\leq M} |L^{\rm res,+}_\ell(u_j)-L^{\rm res,h}_\ell(u_j)|\geq \delta/6\Big)<\e/6.
\end{equation}
For a given $u\in[-M,M]$, choose a $j$ such that $|u_j-u|\leq \delta_1$. Then,
\begin{equation}
\begin{aligned}
|L^{\rm res,+}_\ell(u)-L^{\rm res,h}_\ell(u)|&\leq |L^{\rm res,+}_\ell(u)-L^{\rm res,+}_\ell(u_j)|+|L^{\rm res,+}_\ell(u_j)-L^{\rm res,h}_\ell(u_j)|\\
&+|L^{\rm res,h}_\ell(u_j)-L^{\rm res,h}_\ell(u)|,
\end{aligned}
\end{equation}
which, together with (\ref{eq3.25}) and (\ref{eq3.26}), implies (\ref{eq3.23}) as claimed.

Finally we prove: for any $\e,\delta>0$, there exists a $\ell_0$ (depending on $M$ only) such that for all $\ell\geq\ell_0$,
\begin{equation}\label{eq3.23b}
\Pb\Big(\max_{|u|\leq M} L^{\rm res,+}_\ell(u)-L^{\rm res,{\cal L}}_\ell(u)\geq \delta/2\Big)<\e/2.
\end{equation}
Since the rescaled LPP are defined by linear interpolation between lattice-points, the maximum in (\ref{eq3.23b}) is effectively only on $\Or(\ell^{2/3})$ number of points and thus it is enough to have a bound on $\Pb\Big(L^{\rm res,+}_\ell(u)-L^{\rm res,{\cal L}}_\ell(u)\geq \delta/2\Big)$, which multiplied by $\ell^{2/3}$ goes to $0$ as $\ell\to\infty$.

Without rescaling, since the maximizers from $I(u)$ to $(\ell/4,\ell/4)$ can pass by $I_+(u)$ (and will pass close to it since $I_+(u)$ is essentially on the characteristic line between $I(u)$ and $(\ell/4,\ell/4)$), we have
\begin{equation}\label{eq3.37}
L_{I(u)\to(\ell/4,\ell/4)}\geq L_{I(u)\to I_+(u)}+L_{I_+(u)\to(\ell/4,\ell/4)}.
\end{equation}
A simple algebraic computation shows that
\begin{equation}
\mu((\ell/4,\ell/4)-I(u)) = \mu((\ell/4,\ell/4)-I_+(u)) +\mu(I_+(u)-I(u)) +\Or(1),
\end{equation}
where the $\Or(1)$ term is independent of $\ell$ (and bounded uniformly for $|u|\leq M$).
Thus (\ref{eq3.37}) holds also in the rescaled version up to a correction $\Or(\ell^{-1/3})$:
\begin{equation}\label{eq3.39}
L^{\rm res,+}_\ell(u)-L^{\rm res,{\cal L}}_\ell(u) \leq -\frac{L_{I(u)\to I_+(u)}-\mu(I_+(u)-I(u))}{2^{2/3}\ell^{1/3}}+\Or(\ell^{-1/3}).
\end{equation}
With the choice of $I_+(u)$, $I_+(u)-I(u)=\Or(\ell^{2/3})$ and thus by Proposition~\ref{PropBounds} we have that $\frac{L_{I(u)\to I_+(u)}-\mu(I_+(u)-I(u))}{\sigma \ell^{2/9}}$ converges to a non-degenerate random variable. Since, however, we divide by $\ell^{1/3}\gg \ell^{2/9}$, the right side of (\ref{eq3.39}) will converge to $0$. Take $\ell$ large enough such that the error term $\Or(\ell^{-1/3})$ is bounded by $\delta/4$. Then
\begin{equation}
\begin{aligned}
\Pb\Big(L^{\rm res,+}_\ell(u)-L^{\rm res,{\cal L}}_\ell(u)\geq \delta/2\Big) &\leq
\Pb\Big(\frac{L_{I(u)\to I_+(u)}-\mu(I_+(u)-I(u))}{2^{2/3}\ell^{1/3}}\leq -\delta/4\Big)\\
&= \Pb\Big(\frac{L_{I(u)\to I_+(u)}-\mu(I_+(u)-I(u))}{2^{2/3}\ell^{2/9}}\leq -\delta\ell^{1/9}/4\Big)\\
&\leq C e^{-c(\delta)\ell^{1/6}}
\end{aligned}
\end{equation}
where the last inequality comes from Proposition~\ref{PropBounds}(c). Thus we have obtained the required bound to prove (\ref{eq3.23b}), which was the last piece in the proof of Theorem~\ref{thmSlowDecorrelation}.
\end{proof}

\begin{cor}\label{corTightness}
Consider the LPP from $\{(m,n)\in\Z^2\,|\,n+m=0\}$, the anti-diagonal, and rescale it as
\begin{equation}\label{eq3.19}
L^{\rm res,ad}_\ell(u):=\frac{L_{(u (\ell/2)^{2/3},-u (\ell/2)^{2/3})\to(\ell/4,\ell/4)}-(\ell-2^{2/3}u^2\ell^{1/3})}{2^{2/3}\ell^{1/3}},
\end{equation}
for $u\in\R$ with $u (\ell/2)^{2/3}\in\Z$ and linearly interpolate for the other values of $u$. Then, for any given $M>0$,
the collection $\{L^{\rm res}_\ell\}$ is tight in the space of continuous functions of $[-M,M]$. Furthermore, $L^{\rm res}_\ell$ as a stochastic process in $u\in[M,M]$ converges weakly to the Airy$_2$ process.
\end{cor}
\begin{proof}
This follows from Theorem~\ref{PropTightnessCator} and Theorem~\ref{thmSlowDecorrelation}, since in the latter we can take ${\cal L}$ to be the anti-diagonal as well.
\end{proof}

\section{Monte Carlo simulations}\label{sectSimulation}
Let us recall the variational formula obtained in Theorem~\ref{ThmVarProblem},
\begin{equation}\label{eq1.21}
F^{(\sigma)}(s):=\Pb\left( \max_{u\in\R} \left[\sqrt{2}\sigma B(u)+ {\cal A}_2(u)-u^2\right]\leq s\right).
\end{equation}
While this defines the limit distribution, it seems to be difficult to either deduce more explicit expressions or to run  numerical evaluations. To visualize our result we therefore return to the TASEP. The TASEP is particularly convenient as it is accessible by Monte Carlo simulations.

The TASEP exhibits well-known finite-time corrections. In particular, although the shape of the distribution is accurate even for relatively small times, one generically observes a shift, which is due to the slow $t^{-1/3}$ convergence of the first moment~\cite{FF11,Tak12,TS10,TSSS11}. The speed of convergence of the variance is $t^{-2/3}$. Therefore, in order to minimize the errors due to finite-time corrections, we extrapolate numerically the large time first and second cumulants and scale the numerical data to match the limiting first two cumulants. This procedure gives good results for the special case $\sigma=0$ and $\sigma=1$ for which the cumulants are known analytically as well.

In Figure~\ref{PlotDensities}, we show the probability densities obtained by the Monte Carlo simulations for a subset of values of the $\alpha$ given in Table~\ref{TableCumNumerics}.
\begin{figure}
\begin{center}
\psfrag{0.5}[r]{$0.5$}
\psfrag{0.4}[r]{$0.4$}
\psfrag{0.3}[r]{$0.3$}
\psfrag{0.2}[r]{$0.2$}
\psfrag{0.1}[r]{$0.1$}
\psfrag{0.0}[r]{$0.0$}
\psfrag{-}[ct]{$$}
\psfrag{m2}[ct]{$-2$}
\psfrag{0}[ct]{$0$}
\psfrag{2}[ct]{$2$}
\psfrag{4}[ct]{$4$}
\includegraphics[height=6cm]{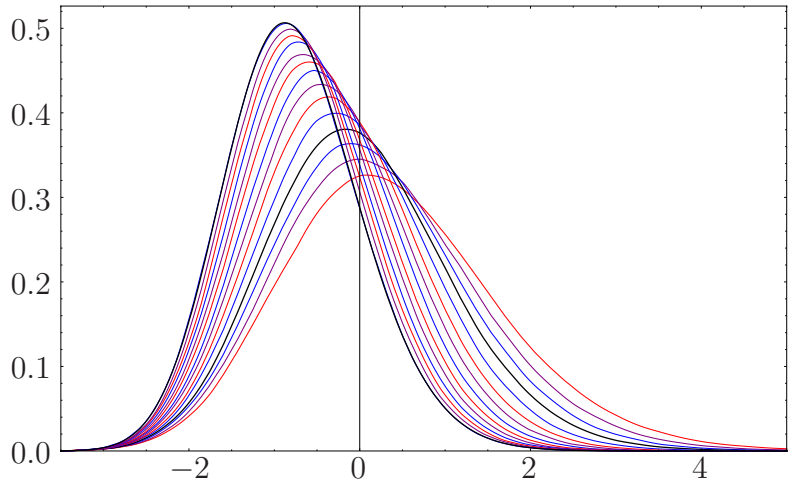}
\caption{Probability densities of $F^{(\sigma)}(s)$ with $\sigma=\sqrt{\alpha/(1-\alpha)}$ from TASEP simulation until time $t_{\rm max}=10^3$ and $10^6$ runs. The different plots corresponds to the values \mbox{$\alpha=0,0.05,0.1,0.15,0.20,0.25,0.3,0.35,0.4,0.45,0.5,0.54,0.58,0.62$}. The left-most black line is the exact rescaled GOE distribution ($\sigma=0$), which overlaps with $\alpha=0$ from the simulations. The black line in the middle is the stationary case ($\sigma=1$).}
\label{PlotDensities}
\end{center}
\end{figure}

When performing a simulation or even a real experiment as in~\cite{Tak12,TS10,TSSS11}, the quantities often measured are not necessarily the probability density function, rather only the first four cumulants. For potential comparison with future numerical and experimental results, we provide below the cumulants obtained through Monte Carlo simulations of TASEP, see Table~\ref{TableCumNumerics}. The values for $\sigma=0$ and $\sigma=1$ are known and they give an indication for the precision of the numerical simulation.
\begin{table}[h!]
\begin{center}
\begin{tabular}{|c|c|c|c|c|c|}
  \hline
$\alpha $ & $\sigma$ & $\kappa_1$=Mean & $\kappa_2$=Variance & $\kappa_3$=3rd cumulant & $\kappa_4$=4th cumulant \\
  \hline
\textbf{0} & \textbf{0} & \textbf{-0.760} & \textbf{0.638} & \textbf{0.149} & \textbf{0.067}\\
0 &  0 & -0.763	& 0.639 & 0.150 & 0.066\\
0.05 & 0.229 & -0.710 & 0.653 & 0.146 & 0.065\\
0.10 & 0.333 &-0.662 & 0.674 & 0.147 & 0.063\\
0.15 & 0.42 &-0.609 & 0.698 & 0.15 & 0.068\\
0.20 & 0.5 & -0.546 & 0.729 & 0.156 & 0.068\\
0.25 & 0.577 & -0.484 & 0.765 & 0.169 & 0.077\\
0.30 & 0.655 & -0.406 & 0.808 & 0.183 & 0.088\\
0.35 & 0.734 & -0.325 & 0.871 & 0.213 & 0.11\\
0.40 & 0.816 & -0.233 & 0.938 & 0.261 & 0.17\\
0.42 & 0.851 & -0.189 & 0.973 & 0.283 & 0.19\\
0.44 & 0.886 & -0.15 & 1.01 & 0.312 & 0.22\\
0.45 & 0.905 & -0.125 & 1.03 & 0.328 & 0.23\\
0.46 & 0.923 & -0.101 & 1.05 & 0.348 & 0.27\\
0.48 & 0.961 & -0.051 & 1.10 & 0.389 & 0.31\\
\textbf{0.5} & \textbf{1} & \textbf{0} & \textbf{1.150} & \textbf{0.444} & \textbf{0.383}\\
0.50 & 1 & 0.002 & 1.15 & 0.430 & 0.36\\
0.52 & 1.04 & 0.056 & 1.21 & 0.487 & 0.43\\
0.54 & 1.08 & 0.115 & 1.27 & 0.552 & 0.51\\
0.56 & 1.13 & 0.181 & 1.34 & 0.632 & 0.63\\
0.58 & 1.18 & 0.247 & 1.42 & 0.721 & 0.73\\
0.60 & 1.22 & 0.326 & 1.51 & 0.836 & 0.92\\
0.62 & 1.28 & 0.407 & 1.62 & 0.969 & 1.1\\
0.64 & 1.33 & 0.494 & 1.73 & 1.12 & 1.3\\
0.70 & 1.53 & 0.822 & 2.23 & 1.87 & 2.7\\
0.80 & 2 &    1.72 &  3.93 & 5.49 & 11.7\\
0.90 & 3 &    3.91 & 10.5 & 28.1 & 102\\
\hline
\end{tabular}
\end{center}
\caption{Table of cumulants obtained numerically by extrapolating the data. The bold numbers are the exactly known numbers. The higher the degree of the cumulant, the lower is the reliability on the numerical extrapolation. The empirical extrapolations are: $\kappa_1(\sigma)=-0.74795 - 0.96976 \sigma + 1.7217 \sigma^{4/3}$, $\kappa_2(\sigma)=0.64268 + 0.0068163 \sigma + 0.50202 \sigma^{8/3}$, $\kappa_3(\sigma)=0.15631 - 0.049956 \sigma + 0.32731 \sigma^4$, and $\kappa_4(\sigma)=0.059476 + 0.012039 \sigma + 0.28315 \sigma^{16/3}$ (the decimals are not though to be all significative).}
\label{TableCumNumerics}
\end{table}

\section{The semi-infinite six-vertex model at its conical (KPZ) point}\label{sectSixVertex}
The planar six-vertex model can be viewed as a statistical mechanics model of interacting up-right lattice paths.
Paths may touch, but they do not cross.
The weight of such a collection of paths is defined by the product of the local weights, which in the canonical labeling are given by
\begin{center}
\psfrag{w1}[c]{$\omega_1$}
\psfrag{w2}[c]{$\omega_2$}
\psfrag{w3}[c]{$\omega_3$}
\psfrag{w4}[c]{$\omega_4$}
\psfrag{w5}[c]{$\omega_5$}
\psfrag{w6}[c]{$\omega_6$}
\includegraphics[height=2cm]{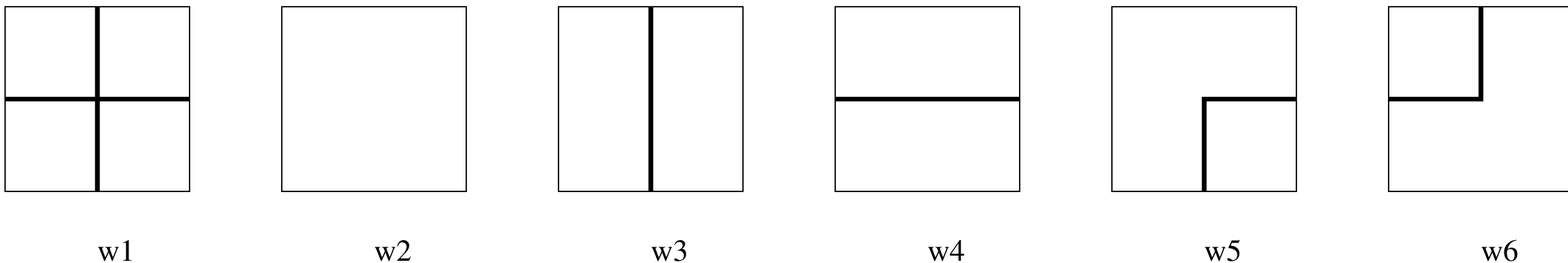}
\end{center}

Much studied is the infinite volume limit with periodic boundary conditions. This results in a translation-invariant Gibbs measure,
which might not be extremal, thereby possibly allowing for either ferro- or anti-ferroelectric ordering. A second popular choice are domain wall
boundary conditions. For a given lattice square, all paths enter at the left edge and exit at the upper edge. There are no paths at the right and lower edge. If paths are interpreted as level lines of a surface over the two-dimensional lattice, then domain wall boundary conditions correspond to a surface boundary which varies on the macroscopic scale. Thus an arctic circle phenomenon is expected~\cite{CP15,CP13b,CPS16b}. In fact, on its free fermion line the six-vertex model with domain wall boundary conditions
is equivalent to the domino tiling of a lattice square rotated by $\pi/4$~\cite{FS06}, also known as  Aztec diamond, which is well-studied~\cite{Jo03,JPS98,
ZJ00}.

In our context, the natural geometry is a semi-infinite lattice. For Ising type models, the influence of boundary fields has been widely
studied~\cite{FS87}, in particular close to criticality. We are not aware of corresponding investigations for the six-vertex model. So far, the results to be described hold only on the stochastic line, which is defined by
\begin{equation}\label{eq4.1}
\vec{\omega} = (1,1,q,p,1-q,1-p),\quad 0\leq p,q\leq1.
\end{equation}
This choice is known as the conical point of the six-vertex model~\cite{BS94}. The name has its origin from studying the free energy, which depends on $\vec\omega$ only through four independent parameters. Using the anti-ferroelectric picture one fixes two temperature-like parameters with $\Delta=(p+q)/(2 p q)>1$ and considers the free energy as a function of the horizontal and vertical electric fields. The free energy has exactly two points satisfying (\ref{eq4.1}) and a conical tip close to such point. The Gibbs measure at the tip exhibits KPZ type fluctuations.

The semi-infinite lattice is obtained by cutting along the anti-diagonal at which boundary conditions are imposed. The Gibbs measure of the semi-infinite lattice is then obtained through the diagonal transfer matrix, which in fact is a Markov chain. We introduce the occupation variables $\eta_{j,t}\in\{0,1\}$, with $0$ in case there is no path at the edge and $1$ in case a path is crossing the edge. We use $j \in \Z$ as the spatial index and $t \in \Z$ as the temporal index. At
two neighboring sites, $j,j+1$, the allowed states are $00,10,01,11$. They are updated according to the transition matrix
\begin{equation}
T^{j,j+1} =
\begin{pmatrix}
1&0&0&0\\
0&1-p&p&0\\
0&q&1-q&0\\
0&0&0&1
\end{pmatrix}.
\end{equation}
The updates of the evenly blocked lattice, $...(01)(23)...$, and the oddly blocked lattice, $...(12)(34)...$, are alternate. Thus, the two step transition probability,
$T$, is defined by
\begin{equation}
 T = T_{\rm e}T_{\rm o}, \quad T_{\rm e} = \bigotimes_{j \in \Z} T^{2j,2j+1},\quad T_{\rm o} = \bigotimes_{j \in \Z} T^{2j-1,2j},
\end{equation}
all acting on $\{0,1\}^\Z$.

Let us consider the alternating product measure $\mu_{(a,b)}=...\otimes\mu_{\rm e} \otimes \mu_{\rm o}\otimes ...$, setting $\mu_{\rm e}(1) = a$,  $\mu_{\rm o}(1) = b$ (say with $\mu_{\rm e}$ on even and $\mu_{\rm o}$ on odd sites). If
\begin{equation}\label{eq4.4}
(1-q)(1-b)a = (1-p) b (1-a),
\end{equation}
then, denoting by $\tau$ the lattice shift,
\begin{equation}
\mu_{(a,b)} T_{*} = \mu_{(a,b)} \circ \tau ,\textrm{ with }*\in\{{\rm e},{\rm o}\}.
\end{equation}
Hence $T$ has a one-parameter family of stationary measures $\mu_{(a^*,b^*)}$, $a^*,b^*$ solution of (\ref{eq4.4}). Thereby we arrive at the Markov chain $\eta_{j,t}$, stationary in both $j$ and $t$.
By construction, this is an  extremal Gibbs measure of the six-vertex model with vertex weights (\ref{eq4.1}). An alternative construction is
explained in~\cite{Agg16}. The density-current relation, $\mathsf{j}(\rho)$, of this process is the solution to
\begin{equation}
\mathsf{j}^2 + ((2-p-q)/(q-p))\mathsf{j}  + \rho(1-\rho) = 0.
\end{equation}
Small perturbations travel with velocity $v = \mathsf{j}'(\rho)$.
The two-point function \mbox{$S(j,t)=\E(\eta_{j,t}\eta_{0,0})-\E(\eta_{j,t})\E(\eta_{0,0})$}, decays exponentially except along the special line $\{j= vt\}$. At distance of order $t^{2/3}$ to this line, according to KPZ scaling theory~\cite{KMH92} (we use the normalisation as in (12) of~\cite{TS12}) one has the scaling form
\begin{equation}\label{eq4.6}
S(j,t) \simeq A (\Gamma t)^{-2/3}f_\mathrm{KPZ}\big((\Gamma t)^{-2/3} (j - vt)\big)
\end{equation}
with $\Gamma = \tfrac12 A^2 |\mathsf{j}''(\rho)|$, where $A$ is the wandering coefficient in the stationary case. In particular,  $S(j,t)$ decays as  $t^{-2/3}$ along $\{j= vt\}$.

If the boundary field $\eta_{j,0}$ satisfies a central limit theorem as in (\ref{eq2.7}), then our theory predicts a scaling behavior identical to (\ref{eq4.6}), except that
$f_\mathrm{KPZ}$ has to be replaced by a different scaling function. In fact, at stationarity, this scaling function is related to the variance of the random variable on the right hand side of (\ref{eq2.9}), see Proposition~4.1 in~\cite{PS01}.

\begin{figure}
\begin{center}
\includegraphics[height=6cm]{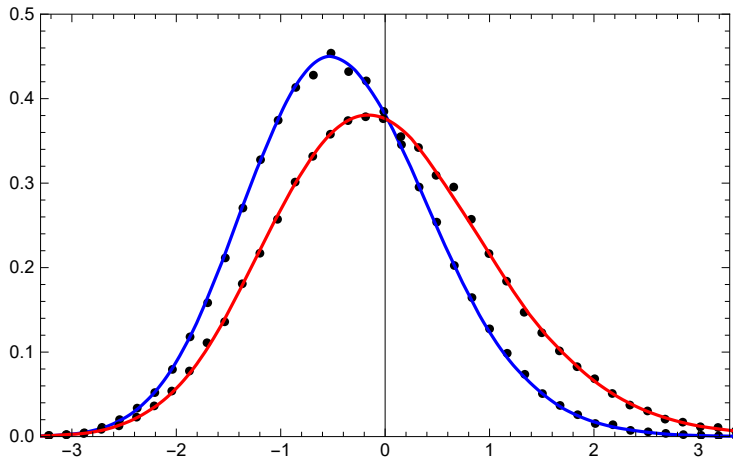}
\caption{Illustration of the universality: for $\sigma^2=\alpha/(1-\alpha)$ with $\alpha=0.3$ (blue) and the stationary benchmark $\alpha=0.5$ (red), the dots are the probabilities for the current fluctuations of the stochastic six-vertex model. Here we choose generic values, $p=0.6$, $q=0.2$, and particle density $\rho=1/2$. The data are scaled according to KPZ scaling theory. We also considered the first order correction (known to generically decay as $t^{-1/3}$~\cite{FF11,Tak12,TS10,TSSS11}), which however resulted only in minimal shifts ($0.011$ resp.\ $0.005$). The solid line is the result of TASEP simulation for the same value of $\sigma$. The number of Monte Carlo runs is $10^5$.}
\label{PlotDensities6vertex}
\end{center}
\end{figure}

To test our predictions, we consider the boundary fields $\mu_{(a,b)}$ with general $0\leq a,b \leq 1$. Then Assumption~\ref{assumptBrownianScaling} holds with
\begin{equation}\label{eq4.8}
\sigma^2 = \tfrac{1}{4\gamma}\big(a(1-a) + b(1-b)\big) \quad \textrm{and}\quad\gamma=\tfrac14\left(a^*(1-a^*)+b^*(1-b^*)\right).
\end{equation}
For the simulation we set $\E(\eta_{j,0}) + \E(\eta_{j+1,0})= 2\rho = 1$, which implies $v=0$ and $a + b =1$ and thus $\sigma^2=a(1-a)/(2\gamma)$. In addition imposing stationarity, the unique solution of (\ref{eq4.4}) in $[0,1]$ is $a=a^*=a_{\rm BR}$ given by
\begin{equation}
a_{\rm BR}=\frac{1 - q -\sqrt{(1-p)(1-q)}}{p - q}.
\end{equation}

Let us denote by $J(t)$ the  current at the origin summed over the time span $[0,t]$. At density $\rho=1/2$, for the scaling coefficient $\Gamma$ we have to set $a=a_{\rm BR}$ and a computation gives $|\mathsf{j}''(\rho)|=|p-q| / \sqrt{(1-p)(1-q)}$. The stationary current is given by $\mathsf{j}=a_{\rm BR}(1-a_{\rm BR})(p-q)/2=a_{\rm BR}-1/2$. Thus, by the KPZ scaling theory, we expect that
\begin{equation}\label{eq4.10}
\lim_{t\to\infty} \Pb\left(\frac{J(t)-\mathsf{j} t}{-(\Gamma t)^{1/3}}\leq s\right)= F^{(\sigma)}(s),
\end{equation}
where the values of $a$ and $b$ for generic $\sigma$ are
\begin{equation}
a=\tfrac12 (1-\sqrt{1-4\sigma^2 a_{\rm BR}(1-a_{\rm BR})}),\quad b=1-a=\tfrac12 (1+\sqrt{1-4\sigma^2 a_{\rm BR}(1-a_{\rm BR})}),
\end{equation}
since (\ref{eq4.8}) has to be satisfied and for $\sigma=1$ also (\ref{eq4.4}). By varying $a$, we  realize the family of universal distributions $F^{(\sigma)}$ with $\sigma\in [0,\sigma_\mathrm{max}]$, where \mbox{$\sigma_\mathrm{max}^2 =(4 a_\mathrm{BR}(1-a_\mathrm{BR}))^{-1}>1$}. The theoretical prediction (\ref{eq4.10}) is checked for a generic value of $(p,q)$ and two values of the diffusion coefficient $\sigma$, namely $\sigma^2=3/7$ and the stationary case. The numerical agreement is very precise as illustrated in Figures~\ref{PlotDensities6vertex} and~\ref{LogPlotDensities6vertex}.

\begin{figure}
\begin{center}
\includegraphics[height=6cm]{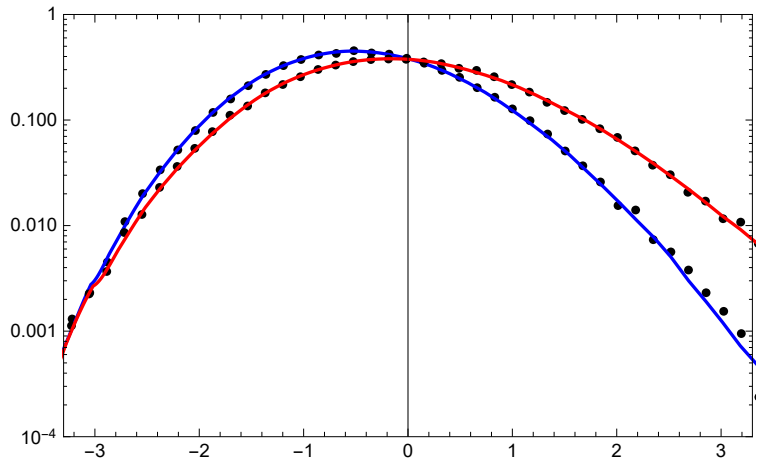}
\caption{Logarithmic plot of the probability densities of Figure~\ref{PlotDensities6vertex}.}
\label{LogPlotDensities6vertex}
\end{center}
\end{figure}

\appendix
\section{Approximate distribution for large $\sigma$}\label{SectLargeSigma}
Since the evaluation of the variational process is difficult, we present an approximate form of (\ref{eq1.21}), which is then compared to Monte Carlo results. For large values of $\sigma$, the fluctuations of the initial Brownian motion should be more important than those from the Airy$_2$ process. A first approximation to (\ref{eq1.21}) is to replace the ${\cal A}_2(u)$ just by ${\cal A}_2(0)$. In this case, the distribution function can be computed analytically. Indeed,
\begin{equation}
F^{(\sigma)}_{\rm appr}(s):=\Pb\left( {\cal A}_2(0)+\max_{u\in\R} \left[\sqrt{2}\sigma B(u)-u^2\right]\leq s\right)
\end{equation}
Since ${\cal A}_2(0)$ and $\max_{u\in\R} \left[\sqrt{2}\sigma B(u)-u^2\right]$ are independent random variables, $F^{(\sigma)}_{\rm appr}$ is the convolution of their distributions. We know that $\Pb({\cal A}_2(0)\leq s)=F_{\rm GUE}(s)$. The distribution function for $\max_{u\in\R} \left[\sqrt{2}\sigma B(u)-u^2\right]$ has been derived by Groeneboom~\cite{Gro89}, see also his more recent paper~\cite{Gro10} describing a numerical evaluation. We denote by $F_{{\rm Gr},c}$ to be the Groeneboom distribution defined by
\begin{equation}
F_{{\rm Gr},c}(s):=\Pb\left(\max_{v\in\R} \left[B(v)-c v^2\right]\leq s\right).
\end{equation}
Then,
\begin{equation}
\Pb\left(\max_{u\in\R} \left[\sqrt{2}\sigma B(u)-u^2\right]\leq s\right) =F_{{\rm Gr},1/\sqrt{2}}\left(\frac{s}{\sqrt{2}\sigma^{4/3}}\right)=:F_{{\rm Gr}}^{(\sigma)}(s).
\end{equation}
As a consequence
\begin{equation}
F^{(\sigma)}_{\rm appr}(s)=\int_\R dx F'_{\rm GUE}(x) F_{{\rm Gr}}^{(\sigma)}(s-x).
\end{equation}
The approximation $F^{(\sigma)}_{\rm appr}(s)$ of $F^{(\sigma)}(s)$ will not be very accurate for small values of $s$, since these corresponds to the realizations of $B$ and ${\cal A}_2$ which are small around $u=0$, and in that case asking that only $B$ is small is not quite the same as $B$ and ${\cal A}_2$ small. However, for larger values of $s$, the approximation should improve, because the events leading to large maximums are due to large values of $B$ and ${\cal A}_2$. Since ${\cal A}_2$ is stationary, these events should be dominated by large fluctuations of the Brownian motion.

A simplifying feature is that $F^{(\sigma)}_{\rm appr}$ is the distribution of a sum of two independent random variables, their cumulants are just the sums of their cumulants. Furthermore, if we denote by $\kappa^{(n)}$, the $n$th cumulant of $F_{{\rm Gr},1/\sqrt{2}}$, then the $n$th cumulant of $F_{{\rm Gr}}^{(\sigma)}$ is given by $2^{n/2}\sigma^{4n/3}\kappa^{(n)}$. The relevant cumulants are reported in Table~\ref{TableCumBasic}.
\begin{table}[h!]
\begin{center}
\begin{tabular}{|c|c|c|c|c|}
  \hline
Distribution & Mean & Variance & 3rd cumulant & 4th cumulant \\
  \hline
$F_{\rm GUE}$ & -1.77108 & 0.8132 & 0.164 & 0.109 \\
$F^{(0)}(\cdot)=\tilde F_{\rm GOE}(2^{2/3}\cdot)$ & -0.76007 & 0.6381 & 0.149 & 0.067 \\
$F^{(1)}=F_{\rm BR}$ & 0 & 1.1504 & 0.444 & 0.383\\
$F_{{\rm Gr},1/\sqrt{2}}$ & 0.8875 & 0.2646 & 0.128 & 0.080 \\
\hline
\end{tabular}
\end{center}
\caption{Table of the first four cumulants for a few distributions.}
\label{TableCumBasic}
\end{table}

After running TASEP simulations for a few values of $\sigma$, surprisingly the probability density plots of $F^{(\sigma)}_{\rm appr}$ and $F^{(\sigma)}$ are already quite close even for fairly small values of $\sigma$. In Figures~\ref{PlotDensitiesLargeSigma} and~\ref{PlotLogDensitiesLargeSigma} we plot the probability densities of $F^{(\sigma)}(s)$ obtained from TASEP simulations and the approximation $F^{(\sigma)}_{\rm appr}(s)$. In order to obtain the probability density of $F^{(\sigma)}_{\rm appr}(s)$, we first use the formula of~\cite{Gro10} to derive a table of values for the density of the Groeneboom distribution $F_{{\rm Gr},1/\sqrt{2}}$ with values in $[0,20]$ taking values every $10^{-3}$, and then we numerically compute the convolution of the rescaled Groeneboom density with the GUE density.
\begin{figure}
\begin{center}
\includegraphics[width=0.48\textwidth]{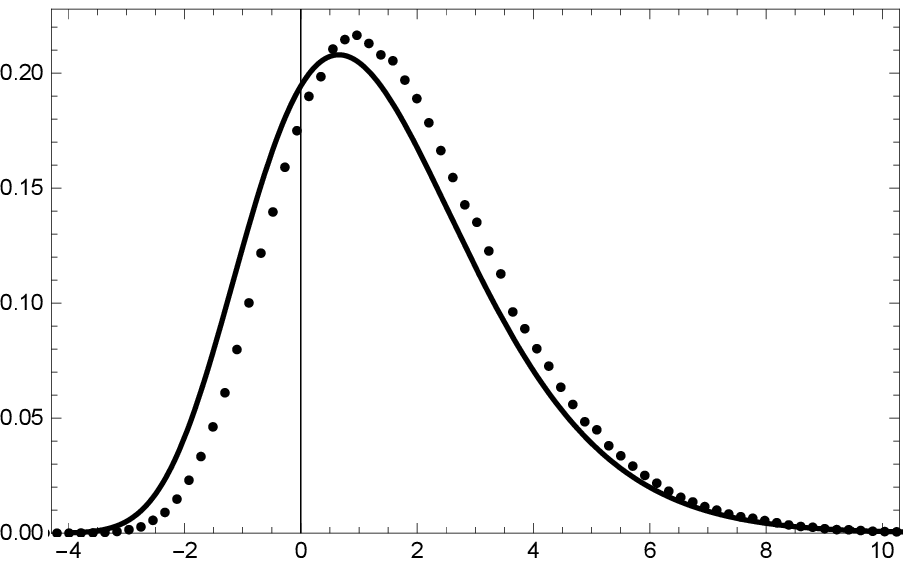} \hfill \includegraphics[width=0.48\textwidth]{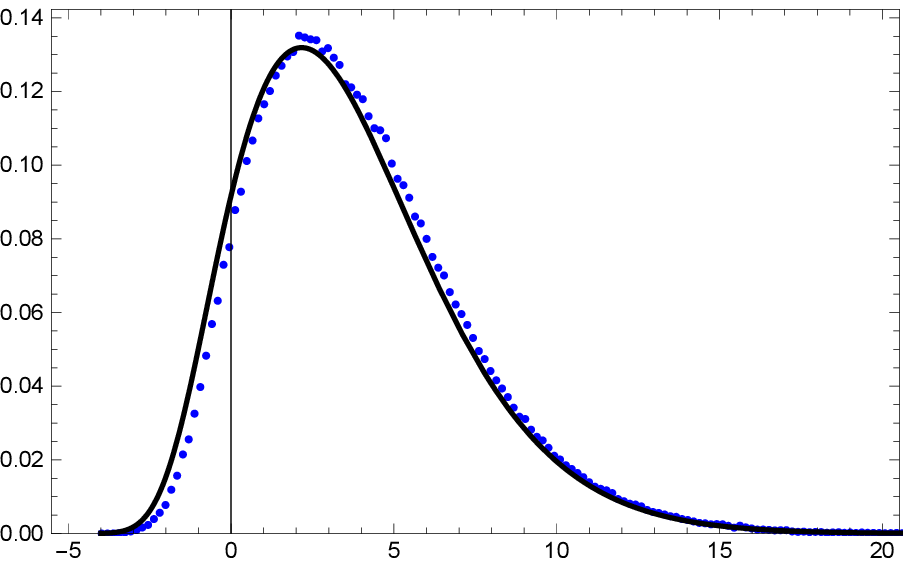} \\[0.5em]
\includegraphics[width=0.48\textwidth]{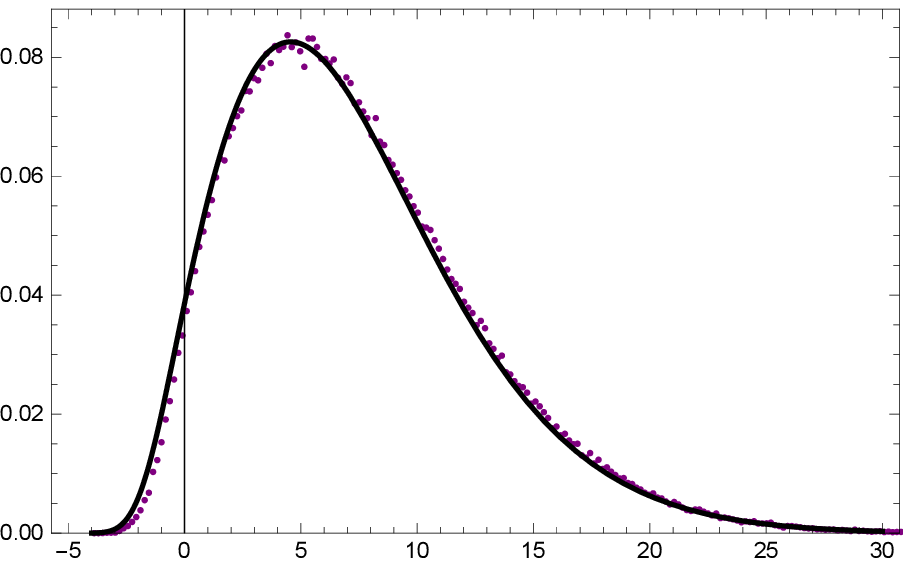} \hfill \includegraphics[width=0.48\textwidth]{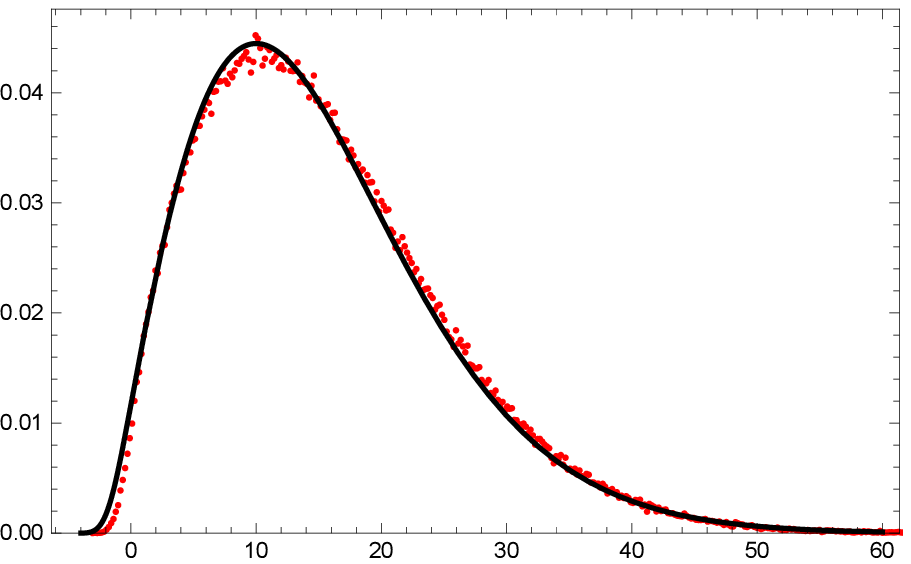}
\caption{Comparison between the probability densities of $F^{(\sigma)}(s)$ and $F^{(\sigma)}_{\rm appr}(s)$: top-left is $\sigma=2.0$, top-right is $\sigma=3.0$, bottom-left is $\sigma=4.36$ and bottom-right is $\sigma=7.0$. The simulation for $\sigma=2.0$ uses maximal time $t_{\rm max}=2000$, while the other $t_{\rm max}=3000$. The number of runs are $5\times 10^5$ for each case.}
\label{PlotDensitiesLargeSigma}
\end{center}
\end{figure}

\begin{figure}
\begin{center}
\includegraphics[width=0.48\textwidth]{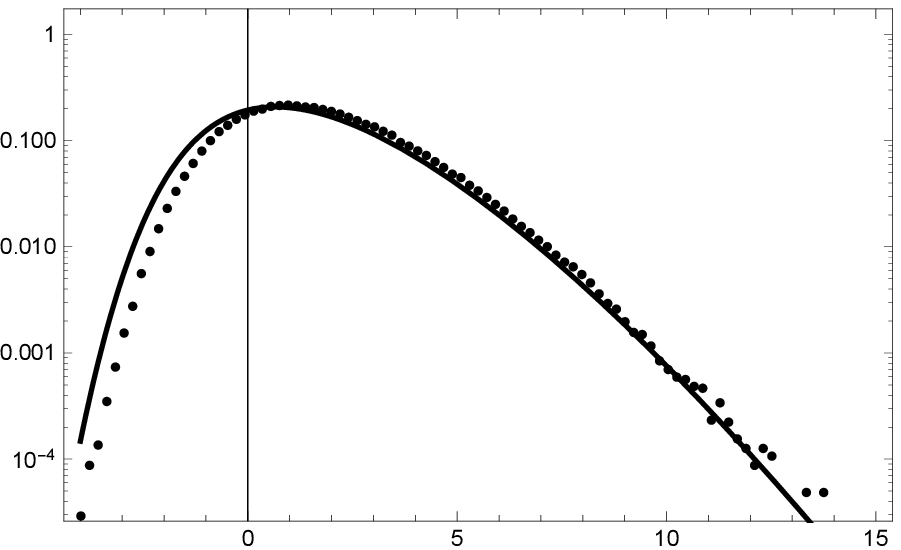} \hfill \includegraphics[width=0.48\textwidth]{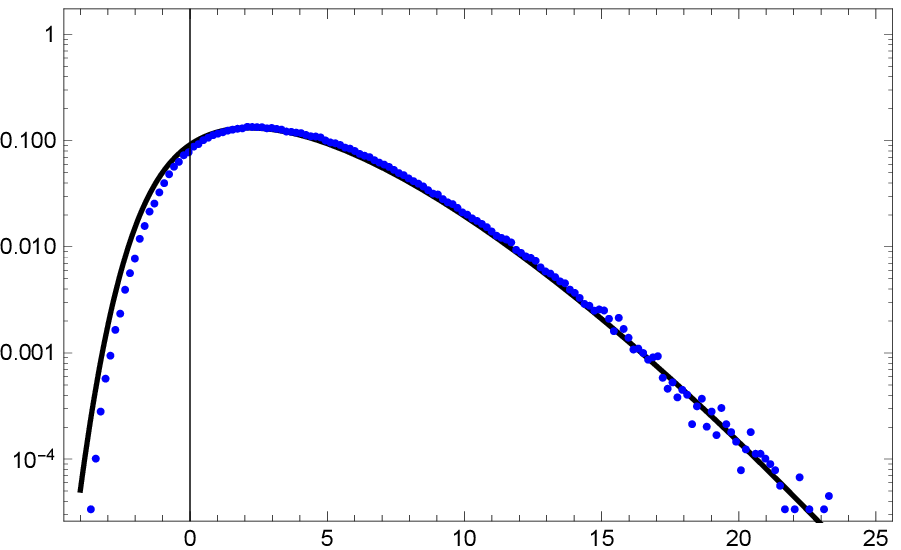} \\[0.5em]
\includegraphics[width=0.48\textwidth]{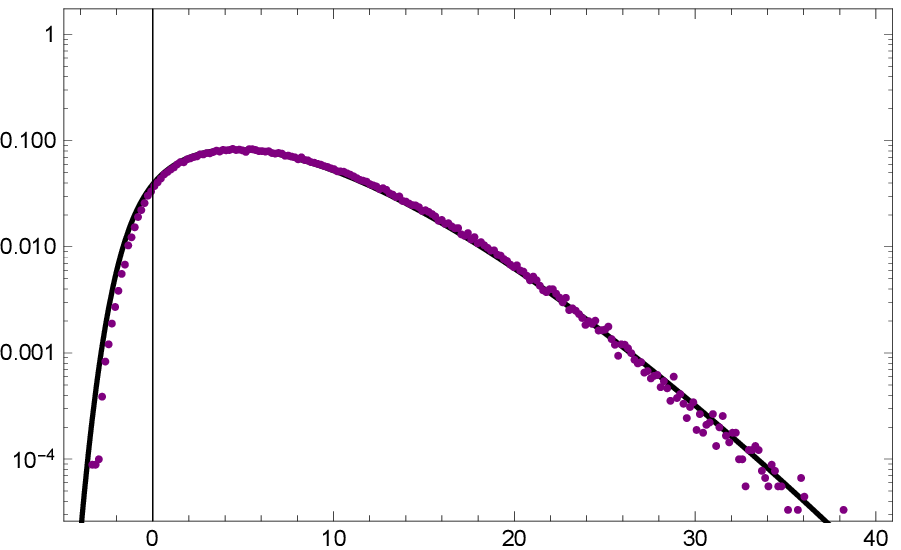} \hfill \includegraphics[width=0.48\textwidth]{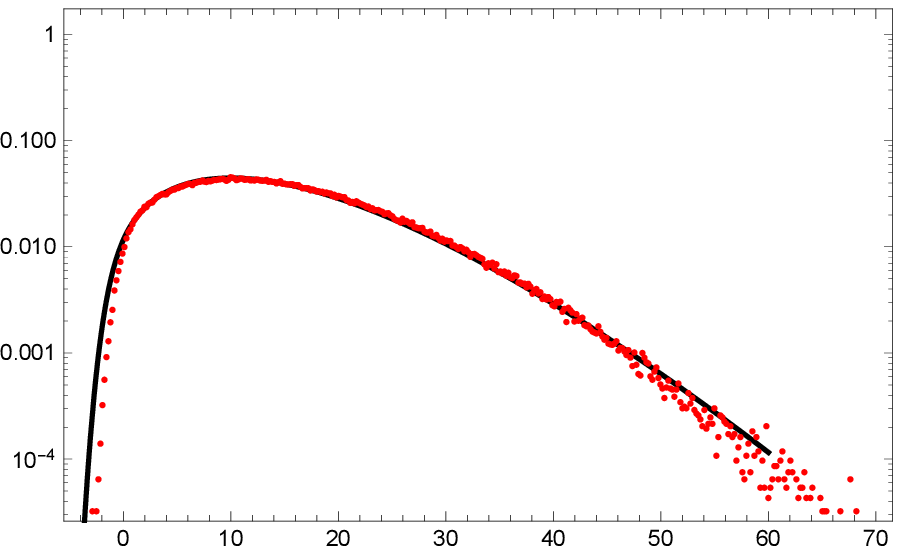}
\caption{Logarithmic plots of the probability densities of Figure~\ref{PlotDensitiesLargeSigma}.}
\label{PlotLogDensitiesLargeSigma}
\end{center}
\end{figure}

%\bibliographystyle{patplain}
%\bibliography{Biblio}

\end{document}